\documentclass[12pt]{article}
\usepackage{amsmath,amssymb}
\usepackage{color,graphicx}

\newcommand{\bu}{\boldsymbol u}
\newcommand{\bv}{\boldsymbol v}
\newcommand{\bw}{\boldsymbol w}
\newcommand{\bz}{\boldsymbol z}
\newcommand{\bt}{\boldsymbol t}
\newcommand{\btt}{\tilde{\boldsymbol t}}

\newcommand{\be}{\boldsymbol e}
\newcommand{\bU}{\boldsymbol U}
\newcommand{\bs}{\boldsymbol s}
\newcommand{\bg}{\boldsymbol g}
\newcommand{\bff}{\boldsymbol f}

\newtheorem{Theorem}{Theorem}
\newtheorem{lema}{Lemma}
\newcounter{remark}
\def\theremark {\arabic{remark}}
\newenvironment{remark}{\refstepcounter{remark}\par\noindent{\bf Remark\ \theremark}\ }{\par}
\newtheorem{Proof}{Proof}

\newenvironment{proof}{\begin{Proof}\rm}{\hfill $\Box$ \end{Proof}}

\usepackage{hyperref}
\usepackage[all]{hypcap}

\title{Analysis of the grad-div stabilization for the time-dependent Navier--Stokes equations
with inf-sup stable finite elements
}
\author{Javier de Frutos\thanks{Instituto de Investigaci\'on en Matem\'aticas (IMUVA),
Universidad de Valladolid, Spain.  Research supported by Spanish MINECO
under grants MTM2013-42538-P (MINECO, ES) and MTM2016-78995-P (AEI/FEDER, UE) (frutos@mac.uva.es)} \and Bosco
Garc\'{\i}a-Archilla\thanks{Departamento de Matem\'atica Aplicada
II, Universidad de Sevilla, Sevilla, Spain. Research supported by
Spanish MINECO under grant MTM2015-65608-P (bosco@esi.us.es)}
\and
Volker John\thanks{Weierstrass Institute for Applied Analysis and Stochastics,
Leibniz Institute in Forschungsverbund Berlin e. V. (WIAS), Mohrenstr. 39, 10117 Berlin, Germany and
Freie Universit\"at Berlin,
Department of Mathematics and Computer Science,
Arnimallee 6, 14195 Berlin, Germany.}
  \and Julia Novo\thanks{Departamento de
Matem\'aticas, Universidad Aut\'onoma de Madrid. Research supported
by Spanish MINECO
under grants MTM2013-42538-P (MINECO, ES) and MTM2016-78995-P (AEI/FEDER, UE) (julia.novo@uam.es)}}
\date{\today}
\begin{document}
\maketitle
\abstract{
This paper studies
inf-sup stable finite element discretizations of
the evolutionary Navier--Stokes equations with
 a grad-div type stabilization.
The analysis covers both
 the case in which the solution is assumed to be smooth
and consequently has to satisfy nonlocal
compatibility conditions as well as the practically relevant situation in which the
nonlocal compatibility conditions are not satisfied.
The constants in the error bounds obtained do not depend
on negative powers of the viscosity.  Taking into account the loss of regularity suffered by the solution
of the Navier--Stokes equations at the initial time in the absence of nonlocal compatibility conditions of the data, error bounds of order $\mathcal O(h^2)$ in space are proved. The analysis is optimal for
quadratic/linear inf-sup stable pairs of finite elements. Both the continuous-in-time case
and the fully discrete scheme with the backward Euler method as time integrator
are analyzed.}


\section{Introduction}

Let $\Omega \subset {\mathbb R}^d$, $d \in \{2,3\}$, be a bounded domain with
polyhedral  and Lipschitz boundary $\partial \Omega$. The incompressible
Navier--Stokes equations model the conservation of linear momentum and the
conservation of mass (continuity equation) by
\begin{align}
\label{NS} \partial_t\bu -\nu \Delta \bu + (\bu\cdot\nabla)\bu + \nabla p &= \bff &&\text{in }\ (0,T)\times\Omega,\nonumber\\
\nabla \cdot \bu &=0&&\text{in }\ (0,T)\times\Omega,\\
\bu(0, \cdot) &= \bu_0(\cdot)&&\text{in }\ \Omega,\nonumber
\end{align}
where $\bu$ is the velocity field, $p$ the pressure, $\nu>0$ the viscosity coefficient,
$\bu_0$ a given initial velocity, and $\bff$ represents external forces acting
on the fluid. The Navier--Stokes equations \eqref{NS} are equipped with homogeneous
Dirichlet boundary conditions $\bu = \boldsymbol 0$ on $\partial \Omega$.

The interest of this paper is the case of small viscosity or, equivalently, high Reynolds
number. To this end, a Galerkin finite element method augmented with a grad-div stabilization term
for \eqref{NS} is considered. Grad-div stabilization
adds a penalty term with respect to the continuity equation to the momentum equation.
It was originally proposed in \cite{FH88} to improve the conservation of mass in
finite element methods. There are a number of papers analyzing the grad-div
stabilization for steady-state problems, e.g., \cite{JJLR14,Ols02,OR04}. On the one hand,
it is known that while grad-div stabilization
improves mass conservation, the computed finite element
velocities are by far not divergence-free \cite{JLMNR15}.
On the other hand, it was observed in the
simulation of turbulent flows that using exclusively grad-div stabilization
resulted in stable simulations, compare \cite[Fig.~3]{JK10} and \cite[Fig.~7]{RL10}.
This observation is one of the motivations for the
present paper: to derive error bounds for the Galerkin finite element method with grad-div
stabilization whose constants do not depend on inverse powers of $\nu$. The analysis will
be performed for pairs of finite element spaces that satisfy a discrete inf-sup condition.
Error bounds with constants independent of $\nu$ were previously  obtained in \cite{los_cuatro_oseen} for the evolutionary Oseen equations. Contrary to the present paper,
the wind velocity in the convective term of the Oseen equations  is divergence-free and
this property considerably simplifies the analysis. Besides extending the
analysis from \cite{los_cuatro_oseen},
more realistic conditions on the initial data are assumed in the present paper, conditions which affect the regularity near the initial time.

An analysis of inf-sup stable elements with divergence-free approximations of the Navier--Stokes equations is presented in  \cite{Scho_Lube}. There,  error bounds independent of
negative powers of $\nu$ were proved for the Galerkin method without any stabilization,
utilizing ideas, e.g., from \cite{los_cuatro_oseen}.
Adding a grad-div stabilization term as in the present paper allows the use of more general, not necessarily divergence-free, finite elements.

Some related works analyzing stabilized finite element approximations to the Navier--Stokes equations
include \cite{Burman_Fer_numer_math}, where the continuous interior penalty method
is studied and \cite{Lube_etalNS,DAL16}, where
the local projection stabilization (LPS) method is studied.
It is discussed in \cite{LAD16} that the case of the Navier--Stokes equations with grad-div stabilization
but without LPS method can be considered as a special case of the analysis presented in \cite{Lube_etalNS}. Notice however that
the error bounds in~\cite{Lube_etalNS} depend explicitly on inverse powers of the viscosity parameter~$\nu$, unless grids are taken sufficiently fine ($h \lesssim \sqrt\nu$, where $h$ is the mesh width), whereas this is not the case in the present paper.
In
\cite{Burman_cmame2015}, error bounds for stabilized finite element approximations to the Navier--Stokes equations are obtained
depending on an exponential factor proportional to the $L^\infty(\Omega)$ norm of the gradient of the large eddies instead of the gradient of the
full velocity $\bu$ in the case that $\Omega$ is the unit square and the boundary conditions are periodic. An analysis of a
fully discrete method based on  LPS  in space and the Euler method in time is carried out
in \cite{Ahmed_etal}. The error bounds in \cite{Ahmed_etal} are not independent of
negative powers of $\nu$. In all these papers, some stabilization terms are added to the
Galerkin formulation. In particular, all these methods, save the method studied in \cite{LAD16}, include a stabilization for the convective term. The aim of  the  present paper consists in deriving error bounds that are
independent of inverse powers of the viscosity parameter for  finite element approximations
that do not include a stabilization of the convective term.

In the present paper, optimal error bounds with constants that do not
depend explicitly  on inverse powers of the viscosity parameter will be  obtained for the $L^2(\Omega)$ norm of the divergence of
the velocity, which measures the closeness of the velocity approximation of being
divergence-free, and the $L^2(\Omega)$ norm of the pressure, assuming that the solution is sufficiently smooth.
In addition,
an error  bound for $\nu^{1/2}$ times the gradient of the velocity is proved.
This error bound is optimal in the viscosity-dominated regime,
although it is a weak term in the convection-dominated regime.
Note that all error bounds might depend implicitly on the viscosity through the
dependency on higher order Sobolev norms of the solution of the continuous problem.

In Section~\ref{Se:smooth}, it will be assumed that the solution satisfies nonlocal compatibility conditions. The analysis is valid for pairs
of inf-sup stable mixed finite elements of any degree. In the case of first order mixed finite elements, the error bound for the pressure can be proved only in two spatial dimensions.

 Due to the increasing use of higher order methods in computational fluid dynamics, the
 question of optimal approximation of the Navier--Stokes equations under realistic
 assumptions of the data has become important. The regularity customarily
 hypothesized in the error analysis for parabolic problems generally cannot be expected
 for the Navier--Stokes
 equations. No matter how regular the initial data are, solutions of the Navier--Stokes equations cannot be assumed to have more than second order spatial derivatives bounded in $L^2(\Omega)$ up to the initial time $t=0$. Higher regularity requires the solution to satisfy some nonlocal compatibility conditions that are  unlikely to be fulfilled in practical situations \cite{heyran0,heyran2}. Taking into account this loss of
 regularity at $t=0$ locally in time, the optimal rate of convergence of the grad-div mixed
 finite element method is studied in Section~\ref{sec:ana_wo_comp}. The analysis of \cite{Ahmed_etal,Lube_etalNS,Burman_Fer_numer_math,Burman_cmame2015,Scho_Lube}
assumes that the solution
satisfies nonlocal compatibility conditions. To the best of
our knowledge, the present paper is the first one where error bounds independent of
the viscosity parameter are proved without those assumptions, and the best bounds that
we obtain are not better than $\mathcal O(h^2)$. In the literature,
\cite{refined,refined_fully,heyran0,heyran2},
error bounds up to~$\mathcal O(h^5 |\log(h)|)$ have been obtained
for both standard and two-grid mixed finite element methods without assuming nonlocal compatibility conditions.
But contrary to the $\mathcal O(h^2)$ bounds in the present paper, the error constants in those $\mathcal O(h^5 |\log(h)|)$  bounds depend on~$\nu^{-1}$.

In Section~\ref{se:fully}, the analysis of the fully discrete case is presented. For the time integration, the
implicit Euler  method is considered. Again, both the regular case and the case in which nonlocal compatibility conditions are not assumed are analyzed.
In this last case, the errors are shown to be
$\mathcal O(h^2|\log(\Delta t)|^{1/2}+(\Delta t)^{1/2})$,
where $\Delta t$ is the size of the time step.

Section~\ref{sec:numres} provides numerical studies supporting the analytical results and a summary finishes the paper.

\section{Preliminaries and notation}\label{sect:prelim}

Throughout the paper, $W^{s,p}(D)$ will denote the Sobolev space of real-valued functions defined on the domain $D\subset\mathbb{R}^d$ with distributional derivatives of order up to $s$ in $L^p(D)$. These spaces are endowed with the usual norm denoted by $\|\cdot\|_{W^{s,p}(D)}$.
 If $s$ is not a positive integer, $W^{s,p}(D)$ is defined by interpolation \cite{Adams}.
 In the case $s=0$, it is $W^{0,p}(D)=L^p(D)$. As it is standard, $W^{s,p}(D)^d$ will be endowed with the product norm and, since no confusion can arise, it will be denoted again by $\|\cdot\|_{W^{s,p}(D)}$. The case  $p=2$  will be distinguished by using $H^s(D)$ to denote the space $W^{s,2}(D)$. The space $H_0^1(D)$ is the closure in $H^1(D)$ of the set of infinitely differentiable functions with compact support
in $D$.  For simplicity, $\|\cdot\|_s$ (resp. $|\cdot |_s$) is used to denote the norm (resp. seminorm) both in $H^s(\Omega)$ or $H^s(\Omega)^d$. The exact meaning will be clear by the context. The inner product of $L^2(\Omega)$ or $L^2(\Omega)^d$ will be denoted by $(\cdot,\cdot)$ and the corresponding norm by $\|\cdot\|_0$. The norm of the space of essentially bounded functions $L^\infty(\Omega)$ will be denoted by $\|\cdot\|_\infty$. For vector-valued functions, the same conventions will be used as before.
The norm of the dual space  $H^{-1}(\Omega)$  of $H^1_0(\Omega)$
is denoted by $\|\cdot\|_{-1}$.
As usual, $L^2(\Omega)$ is always identified
with its dual, so one has $H^1_0(\Omega)\subset L^2(\Omega)\subset H^{-1}(\Omega) $ with compact injection.

Using the function spaces
$
V=H_0^1(\Omega)^d$, and $$ Q=L_0^2(\Omega)=\left\{q\in L^2(\Omega):
(q,1)=0\right\},
$$
the weak formulation of problem (\ref{NS}) is as follows:
Find $(\bu,p)\in V\times Q$ such that for all $(\bv,q)\in V\times Q$,
\begin{equation}\label{eq:NSweak}
(\partial_t\bu,\bv)+\nu (\nabla \bu,\nabla \bv)+((\bu\cdot \nabla) \bu,\bv)-(\nabla \cdot \bv,p)
+(\nabla \cdot \bu,q)=(\boldsymbol f,\bv).
\end{equation}

The Hilbert space
$$
H^{\rm div}=\{ \bu \in L^{2}(\Omega)^d \ \mid \ \nabla \cdot \bu=0, \,
\bu\cdot \mathbf n|_{\partial \Omega} =0 \}$$
will be endowed with the inner product of $L^{2}(\Omega)^{d}$ and the space
$$V^{\rm div}=\{ \bu \in V \ \mid \ \nabla \cdot \bu=0 \}$$
with the inner product of $V$.

Let $\Pi\ :\
L^2(\Omega)^d \rightarrow H^{\rm div}$ be the Leray projector that maps each function in $L^2(\Omega)^d $
onto its divergence-free part (see e.g. \cite[Chapter IV]{Constantin-Foias}. The Stokes
operator in $\Omega$ is given by
$$
A\ : \ \mathcal{D}({A})\subset H^{\rm div} \rightarrow H^{\rm div}, \quad \, {A}=-\Pi\Delta , \quad
\mathcal{D}({A})=H^{2}(\Omega)^{d} \cap V^{\rm div}.
$$
The following Sobolev's embedding \cite{Adams} will be used in the analysis: For  $1\le p<d/s$
let $q$ be such that $\frac{1}{q}
= \frac{1}{p}-\frac{s}{d}$. There exists a positive constant $C$, independent of $s$, such
that
\begin{equation}\label{sob1}
\|v\|_{L^{q'}(\Omega)} \le C \| v\|_{W^{s,p}(\Omega)}, \quad
\frac{1}{q'}
\ge \frac{1}{q}, \quad v \in
W^{s,p}(\Omega).
\end{equation}
If $p>d/s$ the above relation is valid for $q'=\infty$. A similar embedding inequality holds for vector-valued functions.

Let $V_h\subset V$ and $Q_h\subset Q$ be two families of finite element
spaces composed of piecewise polynomials of
degrees at most $k$ and $l$, respectively, that correspond to a family of partitions $\mathcal T_h$ of
$\Omega$ into mesh cells with maximal diameter $h$. In this paper, we will only consider pairs of finite element spaces satisfying
the discrete inf-sup condition,
\begin{equation}\label{LBB}
\inf_{q_h\in Q_h}\sup_{\bv_h\in V_h}\frac{(\nabla \cdot \bv_h,q_h)}{\|\nabla \bv_h\|_0\|q_h\|_0}\ge \beta_0,
\end{equation}
with $\beta_0>0$,  a constant independent of the mesh size $h$. For example, for the
MINI element it is  $k=l=1$ and for the Hood--Taylor element one has
$l=k-1$. Since the error bounds for the pressure depend both on the mixed
finite element used and on the regularity of the solution,
and in general it  will be assumed that  $p\in Q\cap H^k(\Omega)$ with $l\ge k-1$,
in the sequel the error bounds will be written  depending only on $k$.

It will be assumed that the family of
meshes is quasi-uniform and that  the following inverse
inequality holds for each $v_{h} \in V_{h}$, see e.g., \cite[Theorem 3.2.6]{Cia78},
\begin{equation}
\label{inv} \| \bv_{h} \|_{W^{m,p}(K)} \leq C_{\mathrm{inv}}
h_K^{n-m-d\left(\frac{1}{q}-\frac{1}{p}\right)}
\|\bv_{h}\|_{W^{n,q}(K)},
\end{equation}
where $0\leq n \leq m \leq 1$, $1\leq q \leq p \leq \infty$, and $h_K$
is the size (diameter) of the mesh cell $K \in \mathcal T_h$.

The space of discrete divergence-free functions is denoted by
$$
V_h^{\rm div}=\left\{\bv_h\in V_h\ \mid\ (\nabla\cdot \bv_h,q_h)=0\quad \forall q_h\in Q_h \right\},
$$
and by $A_h\ : \ V_h^{\rm div}\rightarrow V_h^{\rm div}$ is denoted the following linear operator
\begin{equation}\label{eq:A_h}
(A_h \bv_h,\bw_h)=(\nabla \bv_h,\nabla \bw_h)\quad \forall \bv_h,\bw_h\in V_h^{\rm div}.
\end{equation}
Note that from this definition, it follows for $\bv_h \in V_h^{\rm div}$ that
$$\|A_h^{1/2}\bv_h\|_0 =
\|\nabla \bv_h\|_0, \quad \|\nabla A_h^{-1/2}\bv_h\|_0 =
\|\bv_h\|_0.
$$
Additionally, two linear operators $C_h\ :\ V_h^{\rm div}\rightarrow V_h^{\rm div}$
and $D_h:L^2(\Omega)\rightarrow V_h^{\rm div}$ are defined by
\begin{align}
(C_h \bv_h,\bw_h)&=(\nabla \cdot \bv_h,\nabla \cdot \bw_h)\quad\,\, \forall \bv_h,\bw_h\in V_h^{\rm div}, \label{eq:C_h}\\
(D_h p,\bv_h)&=-(p,\nabla \cdot \bv_h)\quad\quad\forall \bv_h\in V_h^{\rm div}.\label{eq:D_h}
\end{align}

In what follows,
$\Pi_h^{\rm div}\ : \ L^2(\Omega)^d\rightarrow V_h^{\rm div}$ will denote the so-called discrete Leray projection, which is the
orthogonal projection of $L^2(\Omega)^d$ onto $V_h^{\rm div}$
\begin{equation}\label{eq:Pi_h}
\left(\Pi_h^{\rm div}\bv,\bw_h\right) = (\bv,\bw_h)\quad \forall \bw_h \in V_h^{\rm div}.
\end{equation}
By definition, it is clear that the projection is stable
in the $L^2(\Omega)^d$ norm: $\|\Pi_h^{\rm div}\bv\|_0 \le \|\bv\|_0$ for all
$\bv \in L^2(\Omega)^d$.
The following well-known bound will be used
\begin{equation}\label{eq:cota_pih}
\|(I-\Pi_h^{\rm div})\bv\|_0+h\|(I-\Pi_h^{\rm div})\bv\|_1\le C h^{j+1}\|\bv\|_{j+1}\quad \forall\bv\in V^{\rm div}\cap H^{j+1}(\Omega)^d,
\end{equation}
for~$j=0,\ldots,k$.

Denoting by $\pi_h$ the $L^2(\Omega)$ projection
onto $Q_h$, one has that for $0\le m\le 1$
\begin{equation}\label{eq:pressurel2}
\|q-\pi_hq\|_m\le C h^{j+1-m}\|q\|_{j+1} \quad \forall q\in H^{j+1}(\Omega),
\quad j=0,\ldots,l.
\end{equation}
For simplicity of presentation,  the notation
$\pi_h$ will be used instead of~$\pi_hp$ for the pressure~$p$ in (\ref{NS}).

In the error analysis, the Poincar\'{e}--Friedrichs inequality
\begin{equation}\label{eq:poin}
\|\bv\|_{0} \leq
C\|\nabla \bv\|_{0}\quad\forall \bv\in H_0^1(\Omega)^d,
\end{equation}
will be used.

In the sequel,  $I_h \bu \in V_h$ will denote
the Lagrange interpolant of a continuous function $\bu$. The following bound can be found in \cite[Theorem 4.4.4]{brenner-scot}
\begin{equation}\label{cota_inter}
|\bu-I_h \bu|_{W^{m,p}(K)}\le c_\text{\rm int} h^{n-m}|\bu|_{W^{n,p}(K)},\quad 0\le m\le n\le k+1,
\end{equation}
where $n>d/p$ when $1< p\le \infty$ and $n\ge d$ when $p=1$.

In the analysis, the Stokes problem
\begin{align}\label{eq:stokes_str}
-\nu \Delta \bu+\nabla p&={\bg}\quad\mbox{in }\Omega, \nonumber\\
\bu&=\boldsymbol 0\quad\mbox{on } \partial \Omega,\\
\nabla \cdot \bu&=0\quad\mbox{in } \Omega,\nonumber
\end{align}
will be considered. Let us denote by
$(\bu_h,p_h)\in V_h \times Q_h$  the mixed finite element approximation to \eqref{eq:stokes_str},
given by
\begin{equation*}
\begin{split}
\nu(\nabla \bu_h,\nabla \bv_h)-(\nabla \cdot \bv_h, p_h)&=({\bg},\bv_h)\quad \forall \bv_h\in V_h,\\
(\nabla \cdot \bu_h,q_h)&=0\quad\quad\,\,\,  \forall q_h\in Q_h.
\end{split}
\end{equation*}
Following \cite{girrav,Joh16}, one gets the estimates
\begin{align}\label{eq:cota_stokes_v1}
\|\bu-\bu_h\|_1&\le  C \left(\inf_{\bv_h\in V_h}\|\bu-\bv_h\|_1+\nu^{-1}\inf_{q_h\in Q_h}\|p-q_h\|_0\right),\\
\|p-p_h\|_0&\le C\left( \nu \inf_{\bv_h\in V_h}\|\bu-\bv_h\|_1+
\inf_{q_h\in Q_h}\|p-q_h\|_0\right),\label{eq:cota_stokes_pre}\\
\|\bu-\bu_h\|_0&\le Ch\left(\inf_{\bv_h\in V_h}\|\bu-\bv_h\|_1 + \nu^{-1}\inf_{q_h\in Q_h}\|p-q_h\|_0\right).
\label{eq:cota_stokes_v0}
\end{align}
It can be observed that the error bounds for the velocity depend on negative powers of $\nu$.

For the analysis, it will be advantageous to use a projection of $(\bu,p)$ into $V_h \times Q_h$
with uniform in $\nu$, optimal, bounds for the velocity. In \cite{los_cuatro_oseen} a projection with this property
 was introduced.
Let  $(\bu,p)$ be the solution of the Navier--Stokes equations \eqref{NS} with $\bu\in V\cap H^{k+1}(\Omega)^d$,
$p \in Q\cap H^{k}(\Omega)$,
$k\ge 1$,
and observe that $(\bu,0)$ is the solution of  the Stokes problem \eqref{eq:stokes_str} with
 right-hand side
\begin{equation}\label{eq:stokes_rhs_g}
{\bg}={\bff}-\partial_t \bu-(\bu\cdot \nabla)\bu-\nabla p.
\end{equation}
Denoting the corresponding
Galerkin approximation in $V_h\times Q_h$ by $({\bs_h},l_h)$, one obtains from
\eqref{eq:cota_stokes_v1}--\eqref{eq:cota_stokes_v0}
\begin{align}\label{eq:cotanewpro}
\|\bu-{\bs}_h\|_0+h\|\bu-{\bs}_h\|_1&\le C h^{j+1}\|\bu \|_{j+1},\quad 0\le j\le k,\\
\label{eq:cotanewpropre}
\|l_h\|_0&\le C \nu h^{j}\|\bu\|_{j+1},\quad \hphantom{0}0\le j\le k,
\end{align}
where the constant $C$ does not depend on $\nu$.

\begin{remark}\label{rem:estimate_partial_t}
Assuming the necessary smoothness in time and considering \eqref{eq:stokes_str} with
$$
{\bg} = \partial_t\left({\bff}-\partial_t \bu-(\bu\cdot \nabla)\bu-\nabla p\right),
$$
one can derive an error bound of the form \eqref{eq:cotanewpro}
also for $\partial_t (\bu-{\bs}_h)$. One can proceed similarly for higher order
derivatives in time. In Section~\ref{sec:ana_wo_comp}, where boundedness of derivatives up to~$t=0$ is not assumed,
the bound~\eqref{eq:cotanewpro} is also valid, but then the quantities assumed to be bounded
up to~$t=0$ are
$t^{(j-1)/2}\|\bu(t)\|_{j+1}$,   $t^{(j+1)/2}\|\partial_t\bu(t)\|_{j+1}$,
$t^{(j+3)/2}\|\partial_{tt}\bu(t)\|_{j+1}$, etc. Note that for a given $t_0>0$, the assumptions in the present section hold for~$t\ge t_0$,
and those of~Section~\ref{sec:ana_wo_comp} for~$0\le t\le t_0$.
\end{remark}

Following \cite{chenSiam},
one can also obtain the following bound for ${\bs}_h$
\begin{align}
\|\nabla (\bu-{\bs}_h)\|_\infty&\le C\|\nabla \bu\|_\infty \label{cotainfty1},
\end{align}
where $C$ does not depend on $\nu$.

The method that will be studied for the approximation of the solution of
the Navier--Stokes  equations
(\ref{NS}) is obtained by adding to the Galerkin equations a control of the divergence
constraint (grad-div stabilization). More precisely,
the following grad-div method will be considered: Find $(\bu_h,p_h):(0,T]\rightarrow V_h\times Q_h$ such that
\begin{equation}\label{eq:gal}
\begin{split}
(\partial_t\bu_h,\bv_h)+\nu(\nabla \bu_h,\nabla \bv_h)&+b(\bu_h,\bu_h,\bv_h)-(p_h,\nabla \cdot \bv_h,)\\
&+(\nabla\cdot \bu_h,q_h)+\mu(\nabla \cdot \bu_h,\nabla \cdot \bv_h) =(\boldsymbol f,\bv_h),
\end{split}
\end{equation}
for all $(\bv_h,q_h)\in V_h\times Q_h$,
with $\bu_h(0)=I_h \bu_0$.
Here, and in the rest of the paper,
$$
b(\bu,\bv,\bw)=(B(\bu,\bv),\bw)\quad \forall\bu,\bv,\bw\in H_0^1(\Omega)^d,
$$
where,
$$
B(\bu,\bv)=(\bu\cdot \nabla )\bv+\frac{1}{2}(\nabla  \cdot\bu)\bv\quad \forall \bu,\bv\in H_0^1(\Omega)^d
$$
Notice the well-known property
\begin{equation}
\label{eq:skew}
b(\bu,\bv,\bw)= - b(\bu,\bw,\bv)\quad \forall\bu,\bv,\bw\in V,
\end{equation}
such that, in particular, $b(\bu,\bw,\bw) = 0$ for all $\bu, \bw\in V $.

\section{The regular continuous-in-time case}
\label{Se:smooth}

In this section, error bounds for the continuous-in-time discretization are derived for
the case in which regularity up to time $t=0$ is assumed. Some of the
lemmas are written in such a way that can also be applied in Section~\ref{sec:ana_wo_comp}
for the analysis of the situation without compatibility assumptions.

\subsection{Error bound for the velocity}

Using test functions in $V_h^{\rm div}$ and applying  definitions \eqref{eq:A_h}--\eqref{eq:Pi_h},
one finds that  (\ref{eq:gal}) implies that $\bu_h$ satisfies
\begin{equation}
\label{eq:galp}
\partial_t\bu_h+\nu A_h \bu_h+B_h(\bu_h,\bu_h)+\mu C_h \bu_h=
\Pi_h^{\rm div}{\boldsymbol f},
\end{equation}
where
$$
B_h(\bu,\bv)=\Pi_h^{\rm div}B(\bu,\bv),\quad \bu,\bv\in H_0^1(\Omega)^d,
$$
and $C_h$ is defined in~(\ref{eq:C_h}).
Notice that $\Pi_h^{\rm div}$ can be extended from $L^2(\Omega)^d$ to~$H^{-1}(\Omega)$ in such a way that
$
B_h(\bu,\bv)$
 is well defined for~$\bu,\bv\in H^1_0(\Omega)^d$.

The following two lemmas will be used in Sections~\ref{Se:smooth}  and~\ref{sec:ana_wo_comp}.

\begin{lema}\label{le:stab}
Let $\bw_h\ :\ [0,T]\rightarrow V_h^{\rm div}$ be an arbitrary function piecewise differentiable with respect time. Let $\bu_h$ be the mixed finite
 element approximation to the velocity defined in (\ref{eq:galp}). Define the truncation errors
 $\bt_{1,h}\ :\ [0,T]\rightarrow V_h^{\rm div}$
and $t_2\ : \ [0,T]\rightarrow L^2(\Omega)$ such that the following equation is satisfied
\begin{equation}
\label{eq:law}
\partial_t\bw_h+\nu A_h \bw_h+B_h(\bw_h,\bw_h)+\mu C_h \bw_h=\Pi_h^{\rm div}{\boldsymbol f} +\bt_{1,h}-D_h t_2,
\end{equation}
where $D_h$ has been defined in (\ref{eq:D_h}). Then, if the function
\begin{equation}\label{lag}
g(t)=1+2\|\nabla\bw_h(t)\|_\infty+\frac{\|\bw_h(t)\|_\infty^2}{2\mu}
\end{equation}
is integrable in $(0,T)$, i.e., $\nabla\bw_h \in L^1(0,T;L^\infty)$ and
$\bw_h \in L^2(0,T;L^\infty)$,
the error $\be_h=\bu_h-\bw_h$ can be bounded as follows
\begin{align*}
\|\be_h(t)\|_0^2+\int_0^t e^{K(t,s)}&\left(2\nu\|\nabla \be_h(s)\|_0^2+{\mu}\|\nabla \cdot \be_h(s)\|_0^2\right)~ds\\
&\le
e^{K(t,0)}\|\be_h(0)\|_0^2+\int_0^t e^{K(t,s)}
\left(\|\bt_{1,h}\|_0^2 +\frac{2}{\mu}\|t_2\|_0^2\right)~ds,
\end{align*}
where
$$
K(t,s)=\int_{s}^t \left(1+2\|\nabla\bw_h\|_\infty + \frac{
\| \bw_h\|_\infty^2}{2\mu}\right)\,dr.
$$
\end{lema}

\begin{proof} Subtracting~(\ref{eq:galp}) from~(\ref{eq:law}), taking the
inner product with $\be_h \in V_h^{\rm div}$, and performing some standard computations yields
\begin{equation}\label{eq:err1aa}
\begin{split}
\frac{1}{2}\frac{d}{dt}\|\be_h\|_0^2+\nu\|\nabla \be_h\|_0^2+\mu\|\nabla \cdot \be_h\|_0^2&+b(\bw_h,\bw_h,\be_h)-b(\bu_h,\bu_h,\be_h)\\
&=(\bt_{1,h},\be_h)+(t_2,\nabla \cdot\be_h).
\end{split}
\end{equation}
Observe that
\begin{equation}
\label{eq:un_nl1}
b(\bw_h,\bw_h,\be_h)-b(\bu_h,\bu_h,\be_h)=
-b(\be_h,\bw_h,\be_h)-b(\bu_h,\be_h,\be_h)=-b(\be_h,\bw_h,\be_h),
\end{equation}
where in the last step it was used that, due to (\ref{eq:skew}),
$b(\bu_h,\be_h,\be_h)=0$. Applying H\"older's inequality one finds
\begin{align}
|b(\be_h,\bw_h,\be_h)|&\le \|\nabla \bw_h\|_\infty\|\be_h\|_0^2+\frac{1}{2}\|\nabla \cdot \be_h\|_0\|\bw_h\|_\infty\|\|\be_h\|_0\nonumber\\
&\le
\|\nabla \bw_h\|_\infty\|\be_h\|_0^2+\frac{\mu}{4}\|\nabla \cdot \be_h\|_0^2+\frac{\|\bw_h\|_\infty^2}{4\mu}\|\be_h\|_0^2.
\label{eq:un_nl2}
\end{align}
Thus, from~\eqref{eq:err1aa}, using the Cauchy--Schwarz and Young's inequalities,
taking into account the definition of function $g$ in (\ref{lag}), and rearranging terms,
it follows that
\begin{equation*}
\frac{d}{dt}\|\be_h\|_0^2+2\nu\|\nabla \be_h\|_0^2+\mu\|\nabla \cdot \be_h\|_0^2\le g\|\be_h\|_0^2+\|\bt_{1,h}\|_0^2+ \frac{2}{\mu}\|t_2\|_0^2.
\end{equation*}
Multiplying by the integrating factor $\exp(-K(t,0))$ and integrating in time, the result follows
in a standard way.
\end{proof}

The following lemma will be used in the proof of the main results of the paper.

\begin{lema}\label{le:nonlin} There exists a positive constant $C$ such that
for any $\bw_h\in V_h$ and $\bv\in V\cap H^2(\Omega)^d$ the following bound holds
$$
\left\| B(\bw_h,\bw_h)-B(\bv,\bv)\right\|_0
 \le C\left( \|\bw_h\|_\infty
+\|\nabla \cdot \bw_h\|_{L^{2d/(d-1)}(\Omega)}+\left\|\bv\right\|_2\right)
\|\bw_h-\bv\|_{1}.
$$
\end{lema}
\begin{proof}
Since $V_h  \subset L^\infty(\Omega)$ and by the well known Sobolev
embedding  $H^2(\Omega) \subset L^\infty(\Omega)$ (see e.g., \cite{Adams}),  it follows that
$ B(\bw_h,\bw_h), B(\bv,\bv) \in L^2(\Omega)$. Then, the application of the H\"older
inequality yields
\begin{align*}
\| B(\bw_h,\bw_h)&-B(\bv,\bv)\|_0=\|B(\bw_h,\bw_h-\bv)+
B(\bw_h-\bv,\bv)\|_0\\
&\le \|\bw_h\|_\infty
\|\bw_h-\bv\|_1+\frac{1}{2}\|\nabla \cdot \bw_h\|_{L^{2d/(d-1)}(\Omega)}\|\bw_h-\bv\|_{L^{2d}(\Omega)}\\
&+\|\bw_h-\bv\|_{L^{2d}(\Omega)}\|\nabla \bv\|_{L^{2d/(d-1)}(\Omega)}+\frac{1}{2}\|\nabla \cdot (\bw_h-\bv)\|_0\|\bv\|_\infty.
\end{align*}
The statement of the lemma follows from \eqref{sob1}.
\end{proof}

The proof of the error estimate is based on the comparison of the
Galerkin approximation to the velocity $\bu_h$ in \eqref{eq:gal} with
the approximation ${\bs}_h$ defined at the end of Section~\ref{sect:prelim}. The
pair $({\bs}_h,l_h)\in V_h\times Q_h$ solves
\begin{equation}\label{eq:s_h}
\begin{split}
\nu(\nabla {\bs}_h,\nabla  \bv_h)-(l_h,&\nabla \cdot \bv_h)-(\nabla \cdot {\bs}_h,q_h)\\
 &=({\boldsymbol f},\bv_h)-(\partial_t\bu,\bv_h)
 -b(\bu,\bu,\bv_h)+(p,\nabla \cdot \bv_h).
\end{split}
\end{equation}
Adding and subtracting terms gives
\begin{equation*}
\begin{split}
\partial_t\bs_h+\nu A_h {\bs}_h+B_h({\bs}_h,{\bs}_h)&+\mu C_h {\bs}_h=
\Pi_h^{\rm div}{\boldsymbol f}-\Pi_h^{\rm div}(\partial_t\bu-\partial_t\bs_h)\\
& -\left(B_h(\bu,\bu)-B_h({\bs}_h,{\bs}_h)\right)
-D_h (p-\pi_h)+\mu C_h ({\bs}_h).
\end{split}
\end{equation*}
Taking into account (\ref{eq:C_h}) and $\nabla \cdot \bu=0$, one can see that
Lemma~\ref{le:stab} can be applied  with $\bw_h={\bs}_h$,
$\bt_{1,h}=\Pi_h^{\rm div}(\boldsymbol\tau_1+\boldsymbol\tau_2)$, and $t_2=\tau_3+\tau_4$,
where
\begin{equation}\label{eq:tau3}
\begin{split}
\boldsymbol\tau_1&=\partial_t\bu-\partial_t\bs_h,\quad
\boldsymbol\tau_2=B(\bu,\bu)-B({\bs}_h,{\bs}_h),\\
\tau_3&=p-\pi_h,\quad\quad\tau_4=\mu(\nabla \cdot({\bs}_h-\bu)).
\end{split}
\end{equation}

Let $\bu$ satisfy the hypothesis in Theorem~\ref{th:velocity_smooth}
below.
In order to apply Lemma~\ref{le:stab}, the integrability in $(0,T)$ of the function $g$ defined
in (\ref{lag}), with $\bw_h=\bs_h$, has to be proved. To this end, it will be shown
that the two terms
$\| {\bs_h}\|_\infty^2$ and $ \|\nabla {\bs}_h\|_\infty$
are bounded by an integrable function in $(0,T)$.
For the latter, one can simply apply (\ref{cotainfty1}). For the former term,
one first observes that from the assumed regularity of~$\bu$ it follows
that $\bu$ is continuous and, hence, $\|I_h(\bu)\|_\infty\le C  \|\bu\|_\infty$ for some $C>0$. Then, one can write
$$
\|{\bs}_h\|_\infty\le \|{\bs}_h-I_h(\bu)\|_\infty+\|I_h(\bu)\|_\infty
\le C_{\rm inv}h^{-{d}/2}\|{\bs}_h-I_h(\bu)\|_0+\|\bu\|_\infty,
$$
where in
the last inequality inverse inequality (\ref{inv}) has been applied.
Applying~(\ref{cota_inter}), (\ref{eq:cotanewpro}), and (\ref{sob1}), one gets
\begin{equation}\label{lasinfinito}
\|{\bs}_h\|_\infty\le C  \|\bu\|_2,\quad  \|\nabla {\bs}_h\|_\infty\le C
\|\nabla \bu\|_\infty,
\end{equation}
where the constants are independent of $\nu$.

Thus, by applying Lemma~\ref{le:stab} with $\be_h=\bu_h-\bs_h$,
one obtains
\begin{equation}\label{eq:gron_aux0}
\begin{split}
\|\be_h(t)\|_0^2&+\int_0^t e^{K(t,s)}\left(2\nu\|\nabla \be_h(s)\|_0^2+{\mu}\|\nabla \cdot \be_h(s)\|_0^2\right)~ds\\
&\le
 e^{K(t,0)}\|\be_h(0)\|_0^2+\int_0^t e^{K(t,s)}\left(\|\boldsymbol\tau_1
+\boldsymbol\tau_2\|_0^2+\frac2{\mu}\|\tau_3+\tau_4\|_0^2\right)~ds.
\end{split}
\end{equation}
From (\ref{eq:pressurel2}) and (\ref{eq:cotanewpro})  (see also Remark~\ref{rem:estimate_partial_t})
one gets
\begin{equation}
\label{eq:cota_tau3_tau4}
\left\|\tau_3+\tau_4\right\|_0^2\le Ch^{2k}(\|p\|_{k}^2+
\mu^2\|\bu\|_{k+1}^2),
\end{equation}
and
\begin{equation}
\label{eq:cota_tau1_tau2}
\left\|\boldsymbol\tau_1+\boldsymbol\tau_2\right\|_0^2\le Ch^{2k}\|\partial_t\bu\|_{k}^2+
2\left\|\boldsymbol\tau_2\right\|_0^2.
\end{equation}
For $\boldsymbol\tau_2$, the application of  Lemma~\ref{le:nonlin} gives
\begin{equation}
\label{eq:cota_tau2_aux}
\|\boldsymbol\tau_2\|_0\le C\left( \|{\bs}_h\|_\infty
+\|\nabla \cdot {\bs}_h\|_{L^{2d/(d-1)}(\Omega)}+\left\|\bu\right\|_2\right)
\|\bu- {\bs}_h\|_{1}.
\end{equation}
To bound $\|\nabla \cdot {\bs}_h\|_{L^{2d/(d-1)}(\Omega)}$, one finds with the inverse inequality
(\ref{inv}) that
$$\|\nabla \cdot {\bs}_h\|_{L^{2d/(d-1)}(\Omega)} \le
Ch^{-1/2}\|\nabla \cdot {\bs}_h\|_0,
$$
and
with  \eqref{eq:cotanewpro} it follows that
\begin{equation}\label{eq:cota_div_2d}
\|\nabla \cdot {\bs}_h\|_{L^{2d/(d-1)}(\Omega)} \le Ch^{-1/2}\|\nabla \cdot (\bu - {\bs}_h)\|_0
\le C h^{1/2} \left\|\bu\right\|_2.
\end{equation}
Altogether, from~(\ref{eq:cota_tau2_aux}), using also~(\ref{lasinfinito}) and
that $h$ is bounded (at least from the diameter of $\Omega$), one obtains
\begin{equation}
\label{eq:auxtau2}
\|\boldsymbol\tau_2\|_0\le C\left\|\bu\right\|_2
\|\bu-{\bs}_h\|_{1}\le C h^k\left\|\bu\right\|_2\left\|\bu\right\|_{k+1}.
\end{equation}

In view of~(\ref{lasinfinito}), one has
$1\le \exp(K(t,s))\le  C\exp(L(T))$
with
\begin{equation} L(T)=\int_0^{T} \left(1+2\left\|\nabla \bu(s)\right\|_\infty+
\frac{\left\|\bu(s)\right\|_2^2}{2\mu}\right)\,ds,
\label{eq:LofT}
\end{equation}
where $C$ is independent of~$\nu$.
From~(\ref{eq:gron_aux0}), (\ref{eq:cota_tau3_tau4})--(\ref{eq:auxtau2})
and taking into account that $\|\be_h(0)\|\le C h^{k}\|\bu\|_{k}$,
one derives the following error estimate for the velocity.

\begin{Theorem}\label{th:velocity_smooth} For $T>0$ let us assume for the solution $(\bu,p)$ of \eqref{eq:NSweak} that
$$\bu\in L^2(0,T,H^{k+1}(\Omega))\cap L^2(0,T,W^{1,\infty}(\Omega))\cap L^\infty(0,T,H^2(\Omega)),$$
$\bu(0)\in H^{\max\{2,k\}}(\Omega)$, $\partial_t\bu\in L^2(0,T,H^{k}(\Omega))$,
and $p\in L^2(0,T,H^{k}(\Omega)/{\mathbb R})$ with $k\ge 1$.
Then there exists a positive constant $C$ depending on
\begin{equation}\label{eq:velocity_smooth}
\left\|\bu(0)\right\|_{k}^2+\int_0^T\left(\frac{\left\|p(t)\right\|_{H^k/{\mathbb R}}^2}{\mu}
+
 \|\partial_t\bu(t)\|_{k}^2+(\mu+\|\bu(t)\|_2^2)\| \bu(t)\|_{k+1}^2\right)\,dt,
\end{equation}
but not directly on inverse powers of $\nu$, such that
the following bound holds for $\be_h=\bu_h-{\bs}_h$
and~$t\in[0,T]$
\begin{equation}\label{eq:error_bound_velo}
\|\be_h(t)\|_0^2+\int_0^t (\nu \|\nabla \be_h(s)\|_0^2+\mu\|\nabla \cdot \be_h(s)\|_0^2)~ds
\le C\exp(L(T))h^{2k},
\end{equation}
where $L(T)$ is defined in~(\ref{eq:LofT}).
\end{Theorem}

\begin{remark}\label{re:H3}
Note that Theorem~\ref{th:velocity_smooth} is formulated for the most common choice of inf-sup
stable finite element spaces where the polynomial degree of the velocity space is larger by one than
the degree of the pressure space. In this situation,
 the constant $C$ in \eqref{eq:error_bound_velo} depends on $\mu^{-1}$ and on $\mu$, see \eqref{eq:velocity_smooth}.  Thus, the asymptotic optimal choice of the stabilization parameter is
$\mu \sim 1$, which is a well-known result for this situation.

For pairs of inf-sup stable spaces with the same polynomial degree, like the MINI element,
the same regularity with respect to the polynomial degree for velocity and pressure is usually assumed
and the estimates for
proving the error bound can be adapted accordingly. In particular, one gets instead of \eqref{eq:cota_tau3_tau4}
\begin{equation*}
\left\|\tau_3+\tau_4\right\|_0^2\le Ch^{2k}(h^2\|p\|_{k+1}^2+
\mu^2\|\bu\|_{k+1}^2),
\end{equation*}
such that equilibrating the two terms containing $\mu$ gives the choice $\mu \sim h$, which is known from the
literature \cite{JJLR14}.
However, also choosing $\mu\sim 1$ or $\mu \sim h^2$ leads for the MINI element to optimal error bounds with constants
independent of $\nu$. Altogether, there is some freedom for the choice of  $\mu$  and
choosing this parameter to be a constant is a valid option also for the MINI element.

\end{remark}

\begin{remark}\label{re:th2}
 By writing
 $$
 (\bu_h-\bu)=\be_h+({\bs}_h-\bu),
 $$
 applying the triangle inequality, Theorem~\ref{th:velocity_smooth}, and (\ref{eq:cotanewpro}), it follows that the bound (\ref{eq:error_bound_velo}) holds true changing $\be_h$ by $\bu_h-\bu$.
\end{remark}

\subsection{Error bound for the pressure}
\label{se:pressure}

The error bound for the pressure will be obtained now using the same arguments as used
in \cite{los_cuatro_oseen}.

Applying the inf-sup condition (\ref{LBB}), substituting in the numerator
\eqref{eq:gal} and \eqref{eq:s_h}, adding and subtracting terms,  and using
the Cauchy--Schwarz inequality, it follows that
\begin{align}\label{eq:pressure_aux}
\beta_0\|p_h-\pi_h\|_0&\le \nu\|\nabla \be_h\|_0+\|B(\bu_h,\bu_h)-B({\bs}_h,{\bs}_h)\|_{-1}+\|\partial_t\be_h\|_{-1}\nonumber\\
&+\mu\|\nabla\cdot \be_h\|_0 +\|\boldsymbol\tau_1\|_{-1}
\\
&+\|\boldsymbol\tau_2\|_{-1}+\|\tau_3\|_0+\|\tau_4\|_0+\|l_h\|_0.\nonumber
\end{align}

Note that, due to \eqref{eq:error_bound_velo}, the presence of the  terms $\nu\| \nabla \be_h\|_0$
and $\mu\| \nabla\cdot\be_h\|$ on the right-hand side of \eqref{eq:pressure_aux}
limits the maximum convergence rate to $\mathcal O(h^k)$.
The same convergence rate is obtained for the term $\|\tau_4\|_0$, which
is estimated with \eqref{eq:cotanewpro}.

The fifth term
is bounded with (\ref{eq:cota_tau1_tau2})
\begin{equation*}
\|\boldsymbol\tau_1\|_{-1}\le
C\|\boldsymbol\tau_1\|_0\le Ch^k\left\|\partial_t\bu \right\|_{k},
\end{equation*}
and the sixth term, using \eqref{eq:auxtau2}, by
\begin{equation}
\|\boldsymbol\tau_2\|_{-1}
\le C\|\boldsymbol\tau_2\|_0\le C h^k \|\bu\|_2
\left\|\bu \right\|_{k+1}.
\label{eq:auxB-1shu}
\end{equation}
For the second term on the right-hand side of (\ref{eq:pressure_aux}),
the skew-symmetry of $b(\cdot,\cdot,\cdot)$ gives
\begin{equation}\label{eq:auxb-1}
\begin{split}
\|B(\bu_h,\bu_h)-B({\bs}_h,{\bs}_h)\|_{-1}&=\sup_{\|\boldsymbol\phi\|_1=1}|b(\bu_h,\be_h,\boldsymbol\phi)+
b(\be_h,{\bs}_h,\boldsymbol\phi)|
\\
&=\sup_{\|\boldsymbol\phi\|_1=1}|b(\bu_h,\boldsymbol\phi,\be_h)+
b(\be_h,\boldsymbol\phi,{\bs}_h)|.
\end{split}
\end{equation}
Using now H\"older's inequality and the Sobolev embedding \eqref{sob1}, one finds the bound
\begin{align*}
|b(\bu_h,\boldsymbol\phi,\be_h)| &\le
\|\be_h\|_0\|\bu_h\|_\infty\|\boldsymbol\phi\|_1+\|\be_h\|_0\|\nabla \cdot \bu_h\|_{L^{2d/(d-1)}(\Omega)}\|\boldsymbol\phi\|_{L^{2d}(\Omega)}\nonumber\\
& \le C \left(\|\bu_h\|_\infty+\|\nabla \cdot\bu_h\|_{L^{2d/(d-1)}(\Omega)}\right)\|\be_h\|_0\|\boldsymbol\phi\|_{1}.
\end{align*}
For the second term on the right-hand of~\eqref{eq:auxb-1}, arguing similarly, one gets
$$
|b(\be_h,\boldsymbol\phi,{\bs}_h)| \le \|\be_h\|_0\| {\bs}_h\|_{\infty}\|\boldsymbol\phi\|_{1}
+ C \|\nabla \cdot \be_h\|_0\|{\bs}_h\|_{L^{2d/(d-1)}(\Omega)}\|\boldsymbol\phi\|_{1},
$$
such that
\begin{align}\label{eq:auxB-1}
\|B(\bu_h,\bu_h)&-B({\bs}_h,{\bs}_h)\|_{-1}\nonumber\\
&\le
 C\left(\|\bu_h\|_\infty+\|\nabla \cdot\bu_h\|_{L^{2d/(d-1)}(\Omega)}+\left\|{\bs}_h\right\|_\infty\right)\|\be_h\|_0 \\
 &\quad +
C\|{\bs}_h\|_{L^{2d/(d-1)}(\Omega)}\|\nabla \cdot \be_h\|_0.\nonumber
 \end{align}
Next, the terms between parentheses will be bounded.
Applying~\eqref{inv} and~\eqref{lasinfinito}, one finds
\begin{equation}\label{eq:pres_00}
\begin{split}
\|\bu_h\|_\infty\le \|\be_h\|_\infty+\|{\bs}_h\|_\infty&\le C h^{-d/2}\|\be_h\|_0+\|{\bs}_h\|_\infty\\
&\le C\left(h^{-d/2}\|\be_h\|_0+\|\bu\|_2\right).
\end{split}
\end{equation}
Recalling \eqref{inv} and \eqref{eq:cota_div_2d} yields
\begin{equation}\label{Ineedit}
\begin{split}
\|\nabla \cdot \bu_h\|_{L^{2d/(d-1)}(\Omega)}&\le \|\nabla \cdot \be_h\|_{L^{2d/(d-1)}(\Omega)}
+\|\nabla \cdot {\bs}_h\|_{L^{2d/(d-1)}(\Omega)}
\\&\le C h^{-1/2}\|\nabla \cdot \be_h\|_0+Ch^{1/2} \|\bu\|_2.
\end{split}
\end{equation}

\begin{remark}
The right-hand side of \eqref{eq:pres_00} is bounded  for $h\to 0$
always for $d=2$. It follows from \eqref{eq:error_bound_velo} that
for $d=3$ the term is bounded for $k\ge 2$. Note that most inf-sup stable pairs
of finite element spaces have velocity spaces which are at least of second order so that this is not a big restriction. On the other hand, one can deduce from (\ref{Ineedit}) and \eqref{eq:error_bound_velo} that the term $\int_0^t \|\nabla \cdot \bu_h(s)\|^2_{L^{2d/(d-1)}(\Omega)}~ds$ is bounded.
\end{remark}

With \eqref{eq:auxB-1} and using in addition \eqref{inv}, \eqref{lasinfinito}, and \eqref{eq:cota_div_2d}, one obtains
\begin{align}
\|B(\bu_h,\bu_h)-B({\bs}_h,{\bs}_h)\|_{-1}
\le\ &
C\left\|\bu\right\|_2
 \left(\|\be_h\|_0+\|\nabla \cdot \be_h\|_0\right)\nonumber \\
 &{}+
C\left( h^{-d/2}\left\|\be_h\right\|_0+h^{-1/2}\|\nabla \cdot \be_h\|_0
\right)\|\be_h\|_0.
\label{eq:auxB-12}
\end{align}

Next, the third term on the right-hand side of (\ref{eq:pressure_aux})  will be bounded.
Arguing as in \cite{los_cuatro_oseen}, it will be shown first that
$\|\partial_t\be_h\|_{-1}$ can be estimated by bounding
$\|A_h^{-1/2} \partial_t\be_h\|_0$. From  \cite[Lemma 3.11]{AGN05} it is known that
\begin{equation}\label{new1}
\|\partial_t\be_h\|_{-1}\le C h\|\partial_t\be_h\|_{0}+C\|A^{-1/2}\Pi\partial_t\be_h\|_{0},
\end{equation}
where $\Pi$ is the Leray projector defined in Section~\ref{sect:prelim}.
Applying  \cite[(2.15)]{AGN05} one gets
\begin{equation}\label{new2}
\|A^{-1/2}\Pi\partial_t\be_h\|_{0}\le C h\|\partial_t\be_h\|_0+\|A_h^{-1/2}\partial_t\be_h\|_0,
\end{equation}
with $A_h$ defined in \eqref{eq:A_h}.
With (\ref{new1}), (\ref{new2}), the symmetry of $A_h$,
and the inverse inequality (\ref{inv}), one obtains
\begin{align*}
\|\partial_t\be_h\|_{-1}&\le
C h\|\partial_t\be_h\|_{0}+C\|A_h^{-1/2}\partial_t\be_h\|_0 \nonumber \\
&= C h\|A_h^{1/2}A_h^{-1/2}\partial_t\be_h\|_0+C\|A_h^{-1/2}\partial_t\be_h\|_0\nonumber\\
&= C h \|\nabla (A_h^{-1/2} \partial_t\be_h)\|_0+C\|A_h^{-1/2}\partial_t\be_h\|_0 \nonumber\\
&\le
C\|A_h^{-1/2}\partial_t\be_h\|_0.
\end{align*}
 Taking into account that $\|A_h^{-1/2} \Pi_h^{\rm div} {\bf g}\|_0\le \|{\bf g}\|_{-1},$ for all ${\bf g}\in L^2(\Omega)^d$, see \cite[(2.16)]{AGN05},  and arguing as in \cite{los_cuatro_oseen}, the following
estimate for $\|A_h^{-1/2}\partial_t\be_h\|_0$ can be derived
\begin{equation}\label{eq:pres_01}
\begin{split}
\|A_h^{-1/2}\partial_t\be_h\|_0&\le \nu\|A_h^{1/2}\be_h\|_0+\|B(\bu_h,\bu_h)-B({\bs}_h,{\bs}_h)\|_{-1}
+C\mu\|\nabla \cdot \be_h\|_0 \\
&+\|\partial_t(\bu-\bs_h)\|_{-1}+\|B(\bu,\bu)-B({\bs}_h,{\bs}_h)\|_{-1}
\\
&+C\mu \|\nabla \cdot (\bu-{\bs}_h)\|_0+C\|p-\pi_h\|_0.
\end{split}
\end{equation}
All velocity-related terms on the right-hand side were already estimated in this section.

The pressure terms in \eqref{eq:pressure_aux} and \eqref{eq:pres_01}
are estimated with \eqref{eq:pressurel2} and \eqref{eq:cotanewpropre}.
Then, arguing exactly as in \cite{los_cuatro_oseen}, one concludes the following estimate.

\begin{Theorem}\label{th:pressure_smooth}
Under the assumptions of Theorem~\ref{th:velocity_smooth} there exists a positive constant $C$ such that the following bound holds
\begin{equation}\label{eq:pressure_est}
\beta_0^2\int_0^T\|(p_h-\pi_h)(t)\|_0^2\ dt\le Ch^{2k},
\end{equation}
where in the case $d=3$ the bound is valid for $k\ge2$.
\end{Theorem}
\begin{remark}
By splitting
$$
p_h-p=(p_h-\pi_h)+(\pi_h-p),
$$
applying the triangle inequality, Theorem~\ref{th:pressure_smooth}, and (\ref{eq:pressurel2}), it follows that the bound (\ref{eq:pressure_est}) holds true replacing $p_h-\pi_h$ by $p_h-p$.
\end{remark}

\section{The continuous-in-time case: analysis without nonlocal compatibility conditions}
\label{sec:ana_wo_comp}

It is well known that, no matter how
regular the data are, solutions of the Navier--Stokes equations cannot be
assumed to have more than second order spatial derivatives bounded
in $L^2(\Omega)$ up to
initial time $t=0$, since higher regularity requires the data to satisfy nonlocal
compatibility conditions which are not likely to happen in practical situations
\cite{heyran0,heyran2}. The analysis of this section takes into account
the lack of regularity at $t=0$.

Along the section it is assumed that
inf-sup stable mixed finite elements of second order are used, for
example the Hood--Taylor element consisting of continuous piecewise quadratic polynomials for the velocity and continuous piecewise linear polynomials for the pressure.


It shall be assumed  that for some $T>0$
\begin{equation}\label{eq:velo_h2_ass}
M_1=\max_{0\le t\le T}\left\|\bu(t)\right\|_1<+\infty,\quad
M_2=\max_{0\le t\le T}\left\|\bu(t)\right\|_2<+\infty.\quad
\end{equation}
Also, according to \cite[Theorems~2.4 and~2.5]{heyran0}, and assuming the right-hand side $\bff$ in (\ref{NS}) is smooth enough, it shall be assumed
that, for $k\ge 2$,
\begin{equation}
M_k=\max_{0\le t\le T}t^{(k-2)/2}\left(\left\| \bu(t)\right\|_{k} + \left\| \partial_t\bu(t)\right\|_{k-2}
+\left\| p(t)\right\|_{H^{k-1}/{\mathbb R}}\right)<+\infty,
\label{eq:u-inf}
\end{equation}
and, for $\ k\ge 3$
\begin{equation}
K_{k}^2=
\int_0^Tt^{k-3}
\bigr(\left\| \bu(t)\right\|_{k}^2 + \left\| \partial_s\bu(t)\right\|_{k-2}^2
+\left\| p(t)\right\|_{H^{k-1}/{\mathbb R}}^2+
\left\| \partial_s p(t)\right\|_{H^{k-3}/{\mathbb R}} ^2\bigl)\,ds<+\infty.
\label{eq:u-int}
\end{equation}

\begin{remark}\label{re:case_k=1}
Observe that in view of Remark~\ref{re:H3}, for the case $k=1$ in
Theorem~\ref{th:velocity_smooth} (which covers the case of the so-called
MINI element)
the constant $C$ in (\ref{eq:velocity_smooth}) and the function $L(T)$ from \eqref{eq:LofT} depend
on~$M_2^2(1+T(\mu^{-1}  + \mu + M_2^2)) + K_3^2$ and~$T+2T^{1/2}K_3+TM_2^2\mu^{-1}/2$,
respectively, where no negative powers
of~$t$ appear. Thus, in the absence of nonlocal compatibility conditions
at~$t=0$,
the analysis of the previous section applies to~the~case $k=1$, but it does not apply to the case $k\ge2$ since~$\left\|\partial_t\bu(t)\right\|_{k}^2$ is not
integrable near $t=0$.
\end{remark}

\subsection{An auxiliary function}\label{sec:aux_fct}

For the analysis, the auxiliary function
$\hat {\bu}_h\ :\ [0,T]\rightarrow V_h$ satisfying
\begin{equation}
\label{atenopv}
 \partial_t\hat {\bu}_{h}+ \nu A_{h} \hat {\bu}_{h}= \Pi_{h}^{\rm div}(
 {\boldsymbol f}-\nabla p) -B_{h}(\bu, \bu),\quad
 \hat {\bu}_h(0)=\Pi_{h}^{\rm div} \bu_0,
\end{equation}
will be considered and the following notations will be used
$$
{\boldsymbol z}_h=\Pi_{h}^{\rm div} \bu- \hat {\bu}_h,\quad
{\boldsymbol\theta}_h=\Pi_{h}^{\rm div} {\bu}-{\bs}_h.
$$
Notice that in view of the triangle inequality, \eqref{eq:cotanewpro}, the approximation property
of the $L^2(\Omega)$ projection, and
\eqref{eq:u-inf}--\eqref{eq:u-int} it follows that
for $0<t\le T$,
\begin{align}
\label{cota_theta_inf}
\left\|\boldsymbol\theta_h(t)\right\|_j&\le
M_3 \frac{h^{3-j}}{t^{1/2}},\quad j=0,1,
\\
\int_0^t (\left\|\boldsymbol\theta_h\right\|_j^2+s^2\left\|\partial_s\boldsymbol\theta_h\right\|_j^2)\,ds
&\le C\left(K_3^2+K_5^2\right) h^{6-2j},\quad j=0,1,
\label{cota_theta_int}
\end{align}
for some positive constant~$C$ independent of~$\nu$.
Observe also that projecting~\eqref{NS} onto~$V_h^{\rm div}$, using the
definition of ${\bs}_h$ with \eqref{eq:stokes_str} and the right-hand side given in
\eqref{eq:stokes_rhs_g},
yields
$$
\Pi_h^{\rm div} \partial_t\bu+\nu A_h {\bs}_h+B_h(\bu,\bu)+\Pi_h^{\rm div} \nabla p=\Pi_h^{\rm div}{\boldsymbol f},
$$
so that
$$
\Pi_h^{\rm div} \partial_t\bu+\nu A_h \Pi_h^{\rm div}\bu=\Pi_h^{\rm div}{\boldsymbol f}-B_h(\bu,\bu)-\Pi_h^{\rm div} \nabla p
+\nu A_h \boldsymbol\theta_h.
$$
Subtracting now~\eqref{atenopv} and applying the commutation of the Leray projection and
the temporal derivative, one finds that the error $\bz_h$
satisfies
\begin{equation}\label{elz}
\partial_t \bz_h+ \nu A_h \bz_h= \nu A_h \boldsymbol\theta_h, \quad \bz_h(0)=\boldsymbol 0.
\end{equation}

\begin{lema}
\label{le:estz_h}
There exists a positive constant $C$ independent of~$\nu$ such that the error
$\bz_h=\Pi_{h}^{\rm div}\bu-\hat\bu_h$ of
the discrete velocity~$\hat \bu_h$ defined by~$(\ref{atenopv})$ satisfies
the following bounds for $0<t\le T$:
\begin{align}
\label{eq:fino-z0}
\left\| \bz_h(t)\right\|_0^2+\int_0^t\left\| \bz_h(s)\right\|_1^2 \,ds
&\le CK_3^2 h^{4},
\\
\label{eq:fino-z-1}
\| A_h^{-1/2}\bz_h(t)\|_0^2+\int_0^t\left\| \bz_h(s)\right\|_0^2 \,ds
&\le CK_3^2 h^{6},
\\
\label{eq:cota-z-low}
\| \bz_h(t)\|_j&\le \frac{C}{t^{1/2}}\left(K_3+K_5+M_3\right) h^{3-j},&& j=0,1.\end{align}
\end{lema}

\begin{proof} Multiplying \eqref{elz} by $\bz_h$, integrating on $\Omega$, applying
the Cauchy--Schwarz inequality, and Young's inequality gives
$$
\frac{1}{2}\frac{d}{dt}\left\|\bz_h\right\|_0^2 +
\nu \left\|\nabla \bz_h\right\|_0^2
=\nu(A_h^{1/2}\bz_h,A_h^{1/2}\boldsymbol\theta_h)\le \frac{\nu}{2}
\left\|\nabla \bz_h\right\|_0^2+ \frac{\nu}{2}
\left\|\nabla \boldsymbol\theta_h\right\|_0^2.
$$
Using integration in time and taking into account that $\bz_h(0)=
{\boldsymbol 0}$, it follows that
$$
\left\| \bz_h(t)\right\|_0^2+
\nu \int_0^t \left\| \nabla\bz_h(s)\right\|_0^2\,ds \le
\nu \int_0^t \left\|\nabla \boldsymbol\theta_h(s)\right\|_0^2\,ds.
$$
Now, applying~(\ref{cota_theta_int})
and the Poincar\'e--Friedrichs inequality \eqref{eq:poin}, the bound \eqref{eq:fino-z0} follows directly.
Repeating these arguments but multiplying by~$A_h^{-1}\bz_h$ instead
of~$\bz_h$ gives \eqref{eq:fino-z-1}.

To prove~\eqref{eq:cota-z-low}, multiply \eqref{elz}
by  $t A_h^{-1}\partial_t \bz_h$ and integrate in $\Omega$ to get
$$
t\left\| A_h^{-1/2}\partial_t \bz_h\right\|_{0}^2 +  \frac{\nu}{2}\frac{d}{dt}
\left(t\left\| \bz_h\right\|_0^2\right) = \nu(t\partial_t \bz_h,\boldsymbol\theta_h) +\frac{\nu}{2}\left\| \bz_h\right\|_0^2.
$$
Integrating between~$0$ and~$t$ and integrating the term arising from $\nu(t\partial_t \bz_h,\boldsymbol\theta_h)$ by parts, one gets
\begin{eqnarray*}
\lefteqn{\int_0^t s\left\| A_h^{-1/2}\partial_s \bz_h\right\|_{0}^2\,ds
+\frac{\nu}{2}t\left\| \bz_h(t)\right\|_0^2}\\
&=&\nu t(\bz_h(t),\boldsymbol\theta_h(t))
-\nu\int_0^t \left(\bz_h,\boldsymbol\theta_h+s\partial_s\boldsymbol\theta_h\right)\,ds +
\frac{\nu}{2}\int_0^t\left\| \bz_h\right\|_0^2\,ds.
\end{eqnarray*}
Applying the Cauchy--Schwarz inequality and Young's inequality
to the first two terms
on the right-hand side  and rearranging terms it follows that
\begin{eqnarray*}\lefteqn{
\int_0^t s\left\| A_h^{-1/2}\partial_s \bz_h\right\|_{0}^2\,ds
+\frac{\nu}{4}t\left\| \bz_h(t)\right\|_0^2}\\
&\le &\nu t\left\|\boldsymbol\theta_h(t)\right\|_0^2+
\nu\int_0^t \left(\left\|\boldsymbol\theta_h\right\|_0^2+s^2\left\|\partial_s\boldsymbol\theta_h\right\|_0^2\right)\,ds+
{\nu}\int_0^t \left\| \bz_h\right\|_0^2\,ds.
\end{eqnarray*}
The bound~\eqref{eq:cota-z-low} for $j=0$ now follows by applying
(\ref{cota_theta_inf})--\eqref{cota_theta_int} and {(\ref{eq:fino-z-1})}.
With the same arguments, but
multiplying by $t\partial_t\bz_h$ instead
of $t A_h^{-1}\partial_t \bz_h$  the bound~\eqref{eq:cota-z-low} for $j=1$
is obtained.
\end{proof}

\begin{remark}\label{re:at_most_3} For piecewise polynomials of degree higher than two, it is possible to obtain higher order bounds,
but with negative powers
of~$\nu$. For example, for piecewise cubics, by
repeating the arguments in the proof of Lemma~\ref{le:estz_h},
it is possible to show that
$$
 \| \bz_h(t)\|_j \le \frac{C}{\nu^{1/2}t}\left(K_4+K_6+M_4\right) h^{4-j},\quad j=0,1,\quad 0\le t\le T,
$$
using as test function $t^2A_h^{-1}\partial_t\bz_h$.
Similar negative powers of~$\nu$ are obtained also with some other
techniques like those in~\cite{refined}. At the moment, it is an open question whether different techniques could be applied to get higher order bounds
with constants independent of inverse powers of $\nu$. For this reason, only
piecewise quadratics for the velocity are considered in this section.
\end{remark}

\subsection{Error bound for the velocity}

Observe that the first equation in~\eqref{atenopv} can be rewritten
as
\begin{align*}
 \partial_t\hat {\bu}_{h}+ \nu A_{h}
 \hat {\bu}_{h}+ B_{h}(\hat\bu_h, \hat\bu_h)
 +\mu C_h\hat\bu_h&= \Pi_{h}^{\rm div}
 {\boldsymbol f}+ \left(B_h(\hat\bu_h,\hat\bu_h)-B_{h}(\bu, \bu)\right)\\
 &\phantom{=}-D_h p
 +\mu C_h\hat\bu_h.
\end{align*}
Lemma~\ref{le:stab} will be applied with~$\bw_h=\hat\bu_h$,
$\bt_{1,h}=\widehat{ \boldsymbol\tau}_2$, and
$t_2=\tau_3+\hat\tau_4$,
where
$$
\widehat{ \boldsymbol\tau}_2 = B_h(\hat\bu_h,\hat\bu_h)-B_h(\bu,\bu),\quad
\hat\tau_4 = \mu \nabla \cdot(\hat\bu_h-\bu),
$$
and where~$\tau_3$ is defined in~(\ref{eq:tau3}).
The application of this lemma requires to
show that both $\| \hat\bu_h\|_\infty^2$
and $\|\nabla\hat\bu_h\|_\infty$ are integrable in $(0,T)$.

To bound~$\| \hat\bu_h\|_\infty$, apply the triangle inequality and
the inverse inequality~(\ref{inv}) to get
\begin{align*}
\| \hat\bu_h\|_\infty &\le C_{\rm inv}h^{-d/2}\| \hat\bu_h -I_h\bu\|_0 +
\| I_h\bu\|_\infty
\\
& \le C_{\rm inv}h^{-d/2}\left(\| \hat\bu_h -\Pi_h^{\rm div}
\bu\|_0 + \| \Pi_h^{\rm div}\bu-\bu\|_0+\| \bu-I_h\bu\|_0\right)+
\| I_h\bu\|_\infty.
\end{align*}
Since
$\left\| I_h\bu\right\|_\infty \le C\left\|\bu\right\|_\infty$, utilizing~\eqref{sob1}
gives
$$
\left\| I_h\bu\right\|_\infty \le  C \left\|\bu\right\|_2 \le CM_2.
$$

Applying \eqref{cota_inter},  \eqref{eq:cota_pih}, and  \eqref{eq:fino-z0} yields
\begin{equation}
\| \hat\bu_h\|_\infty \le C\left( \left(K_3+M_2\right)h^{(4-d)/2} +
M_2\right).
\label{eq:hatu_inf}
\end{equation}

The bound of  $\left\|\nabla \hat \bu\right\|_\infty$ will be shown for
the more difficult and practically more relevant case $d=3$.  With the triangle inequality
and the inverse inequality (\ref{inv}), one obtains
\begin{align}
\| \nabla \hat\bu_h\|_\infty \le &C_{\rm inv}h^{-5/2}
\| \hat\bu_h -{\bs}_h\|_0 +
\| \nabla{\bs}_h\|_\infty
\label{eq:whynot}
\\
 \le &C_{\rm inv}h^{-5/2}\left(\| \hat\bu_h -\Pi_h^{\rm div}
\bu\|_0 + \| \Pi_h^{\rm div}\bu-\bu\|_0+\| \bu-{\bs}_h\|_0\right)+
\| \nabla{\bs}_h\|_\infty.
\nonumber
\end{align}
Applying \eqref{cotainfty1}, \eqref{sob1}, and  \eqref{eq:u-inf} yields
\begin{equation*}
\|\nabla {\bs}_h\|_\infty\le CM_3t^{-1/2}.
\end{equation*}
Arguing as before and applying  \eqref{eq:cota-z-low}, (\ref{eq:cota_pih}), and \eqref{eq:cotanewpro}  gives
\begin{equation}
\| \nabla\hat\bu_h\|_\infty \le Ct^{-1/2}\left((K_3+K_5+M_3)h^{1/2} +
M_{3}\right).
\label{eq:nabla_hatu_inf}
\end{equation}

Thus, from \eqref{eq:hatu_inf}
and \eqref{eq:nabla_hatu_inf} it follows that
\begin{equation*}
\begin{split}
g(t)&=1+2\|\nabla\hat \bu_h(t)\|_\infty+\frac{\left\|\hat \bu_h(t)\right\|_\infty^2}{2\mu}\\
&\le
C\left(\frac{K_3+K_5+M_3}{\tau^{1/2}}+ \frac{(K_3+M_2)^2}{2\mu}
\right)
\le\frac{L_1}{t^{1/2}},
\end{split}
\end{equation*}
where
\begin{equation*}
L_1=C\left(K_3+K_5+M_3\right)+\frac{(K_3+M_2)^2}{2\mu}.
\end{equation*}
Then
\begin{eqnarray}
K(t,s)&=&\int_s^tg(r)\,dr
\le L_1\int_s^t \frac{1}{r^{1/2}}\,dr\le 2L_1 (t^{1/2}-s^{1/2}).
\label{eq:cota_Kts}
\end{eqnarray}
Lemma~\ref{le:stab} with $\bw_h=\hat\bu_h$ gives for $\hat \be_h=\bu_h - \hat \bu_h$
\begin{equation*}
\begin{split}
\|\hat\be_h(t)\|_0^2&+\int_0^t e^{K(t,s)}\left(2\nu\|\nabla\hat \be_h(s)\|_0^2
+{\mu}\|\nabla \cdot \hat\be_h(s)\|_0^2\right)~ds\\
& \le  e^{K(t,0)}\|\be_h(0)\|_0^2+\int_0^t e^{K(t,s)}\left(\left\|
\widehat{ \boldsymbol\tau}_2\right\|_0^2+\frac2{\mu}\left\|\tau_3+\hat\tau_4\right\|_0^2\right)\,ds.
\end{split}
\end{equation*}

To estimate the truncation errors first apply Lemma~\ref{le:nonlin} to get
\begin{equation}
\label{eq:cota_hat_tau2_aux}
\|\widehat{ \boldsymbol\tau}_2\|_0\le C\left( \|\hat\bu_h\|_\infty
+\|\nabla \cdot \hat\bu_h\|_{L^{2d/(d-1)}}+\left\|\bu\right\|_2\right)
\|\bu- \hat\bu_h\|_{1}.\
\end{equation}

Using the triangle inequality, the inverse inequality \eqref{inv}, (\ref{eq:cota_div_2d}),
(\ref{eq:fino-z0}), (\ref{eq:cota_pih}), (\ref{eq:cotanewpro}), and \eqref{eq:velo_h2_ass}
one gets
\begin{align}
\|\nabla \cdot \hat\bu_h&\|_{L^{2d/(d-1)}}\nonumber\\
&\le C_{\mathrm{inv}} h^{-3/2}\| \hat\bu_h-
{\bs}_h\|_0 +\|\nabla\cdot {\bs}_h\|_{L^{2d/(d-1)}} \nonumber\\
&\le Ch^{-3/2}\left(\| \hat\bu_h-
\Pi_h^{\rm div}\bu\|_0+\|\Pi_h^{\rm div}\bu- \bu\|_0+\|\bu-{\bs}_h\|_0\right)
+C h^{1/2} \left\|\bu\right\|_2
\nonumber\\
&\le C h^{1/2}(K_3+\|\bu\|_2) \le C (K_3+M_2).\label{eq:cot_div_hatu_2d}
\end{align}
By
inserting (\ref{eq:cot_div_hatu_2d}) and  (\ref{eq:hatu_inf}) in (\ref{eq:cota_hat_tau2_aux}) it follows that
\begin{equation}\label{eq:auxtau2pv}
\|\widehat{ \boldsymbol\tau}_2\|_0
\le C(K_3+M_2)
\|\bu-\hat\bu_h\|_1.
\end{equation}
As in~(\ref{eq:cota_tau3_tau4}) one also gets
\begin{equation}
\label{eq:cota_tau3_hat_tau4}
\left\|\tau_3+\hat\tau_4\right\|_0^2\le C\left(\|p\|_{H^2/{\mathbb R}}^2h^{4}+
\mu^2\|\bu-\hat\bu_h\|_{1}^2\right).
\end{equation}
To bound $\|\bu-\hat\bu_h\|_1$ in \eqref{eq:auxtau2pv} and  \eqref{eq:cota_tau3_hat_tau4} one
adds and subtracts $\Pi_h^{\rm div} \bu$. Applying then
(\ref{eq:u-int}),  (\ref{eq:cota_pih}), and (\ref{eq:fino-z0})  leads to the
estimate
\begin{align}
\label{eq:consistency}
\int_0^t \left(\|\widehat{ \boldsymbol\tau}_2\|_0^2+ \frac2\mu
\|\tau_3+\hat\tau_4\|_0^2\right)\,ds
&\le  CK_3^2\left(\mu+\mu^{-1}+(K_3+M_2)^2\right)h^4
\nonumber
\\
&= C_0h^4.
\end{align}

Collecting all estimates and applying at the initial time
the triangle inequality, the interpolation estimate \eqref{cota_inter} and (\ref{eq:cota_pih}),
the following theorem is proved.

\begin{Theorem}\label{th:main_non_local}
For $T>0$, assuming the solution $(\bu,p)$ of (\ref{NS}) satisfies (\ref{eq:velo_h2_ass}), (\ref{eq:u-inf}) and (\ref{eq:u-int})  the following
bound holds for the error $\hat\be_h=\bu_h-\hat\bu_h$, $t\in[0,T]$:
\begin{equation}\label{eq:velo_non_local}
\|\hat\be_h(t)\|_0^2+\int_0^t \left(2\nu \|\nabla\hat \be_h(s)\|_0^2+\mu\|\nabla \cdot\hat \be_h(s)\|_0^2\right)~ds
\le e^{K(t,0)}\left(M_2^2+C_0\right)h^{4},
\end{equation}
where $K(t,s)$ is defined in~(\ref{eq:cota_Kts}) and~$C_0$ is the constant in~(\ref{eq:consistency}).
\end{Theorem}

\begin{remark}\label{re:th3}
By decomposing
$$
\bu_h-\bu=\hat\be_h+(\hat\bu_h-\bu)=\hat\be_h-\bz_h+\left(\Pi_h^{\rm div}\bu-\bu\right),
$$
and applying the triangle inequality, Theorem~\ref{th:main_non_local},  (\ref{eq:fino-z0}), and (\ref{eq:cota_pih}), it follows that the bound (\ref{eq:velo_non_local}) holds true changing $\hat \be_h$ by $\bu_h-\bu$.
\end{remark}
\begin{remark}\label{re:whynot} Observe that it is the factor
$h^{-5/2}$ in (\ref{eq:whynot}) that prevents the analysis in the present section to apply to
the case~$k=1$. On the other hand, the analysis
in Section~\ref{Se:smooth} applies to $k=1$
since one compares $\bu_h$
with~${\bs}_h$ for which the bounds~(\ref{lasinfinito}) are
available. The comparison with~${\bs}_h$ in~Section~\ref{Se:smooth}, however,
induces the truncation~error~$\boldsymbol\tau_1$, which, as commented
in~Remark~\ref{re:case_k=1}, prevents the extension of
the analysis in that section to the case~$k>1$.
\end{remark}

\subsection{Error bound for the pressure}

The analysis follows closely that of Section~\ref{se:pressure}.

Again, using the inf-sup condition (\ref{LBB}), a straightforward calculation leads to
\begin{align*}
\beta_0\|p_h-\pi_h\|_0&\le \nu\|\nabla \hat \be_h\|_0+\|B(\bu_h,\bu_h)-B(\hat\bu_h,\hat\bu_h)\|_{-1}+\|\partial_t \hat\be_h\|_{-1}\\
&+\mu\|\nabla\cdot
\hat\be_h\|_0 +\|B(\bu,\bu)-B(\hat\bu_h,\hat\bu_h)\|_{-1}\\
&+\mu\|\nabla\cdot  (\bu-{\hat\bu}_h)\|_0 +\|p-\pi_h\|_0+\|l_h\|_0,
\end{align*}
where $l_h$ denotes here the discrete pressure corresponding to the formulation of \eqref{atenopv}
in $V_h$.
Repeating the arguments used when obtaining~(\ref{eq:auxB-1shu})
and \eqref{eq:auxtau2pv}, one gets
\begin{equation*}
\|B(\bu,\bu)-B(\hat\bu_h,\hat\bu_h)\|_{-1} \le C(K_3+M_2)
\|\bu-\hat\bu_h\|_1.
\end{equation*}
In the same way as for (\ref{eq:auxB-1}), one obtains
\begin{align*}
\|B(\bu_h,\bu_h)-B(\hat\bu_h,\hat\bu_h)\|_{-1}
\le &
 \left(\|\bu_h\|_\infty+\|\nabla \cdot\bu_h\|_{L^{2d/(d-1)}}+\left\|\hat\bu_h\right\|_\infty\right)\|\hat\be_h\|_0\nonumber \\
 &{}+
\|\hat\bu_h\|_{L^{2d/(d-1)}}\|\nabla \cdot \hat\be_h\|_0.
 \end{align*}
Using the inverse inequality \eqref{inv}, \eqref{eq:cot_div_hatu_2d}, \eqref{eq:pres_00},
and \eqref{eq:hatu_inf}  leads to
\begin{align*}
\|B(\bu_h,\bu_h)-B(\hat\bu_h,\hat\bu_h)\|_{-1}
&\le
C(K_3+M_2)
 \left(\|\hat\be_h\|_0+\|\nabla \cdot \hat\be_h\|_0\right) \\
 &+
C\left(h^{-d/2} \left\|\hat\be_h\right\|_0+h^{-1/2}\|\nabla \cdot \hat\be_h\|_0
\right)\|\hat\be_h\|_0.\nonumber
\end{align*}

Arguing as in Section~\ref{se:pressure}, one gets for $\|\partial_t \hat\be_h\|_{-1}$
\begin{align*}
\|\partial_t \hat\be_h\|_{-1}&\le \nu\|A_h^{1/2}\hat\be_h\|_0+\|B(\bu_h,\bu_h)-B(\hat\bu_h,\hat\bu_h)\|_{-1}
+C\mu\|\nabla \cdot \hat\be_h\|_0\\
&+\|B(\bu,\bu)-B(\hat\bu_h,\hat\bu_h)\|_{-1}
\\
&+C\mu \|\nabla \cdot (\bu-\hat{\bu}_h)\|_0+C\|p-\pi_h\|_0.
\end{align*}
All terms on the right-hand side of this estimate have already been bounded. Arguing
like at the end of Section~\ref{se:pressure}, one derives the following estimate.
\begin{Theorem}\label{th:pressure_non_smooth}
Under the assumptions of Theorem~\ref{th:main_non_local} there exists a positive constant $C$ such that the following bound holds
\begin{equation}\label{eq:pressure_est_00}
\beta_0^2\int_0^t\|(p_h-\pi_h)(s)\|_0^2\le Ch^{4}.
\end{equation}
\end{Theorem}
\begin{remark}
By writing
$$
p_h-p=(p_h-\pi_h)+(\pi_h-p),
$$
applying the triangle inequality, Theorem~\ref{th:pressure_non_smooth}, and (\ref{eq:pressurel2}), it follows that the bound (\ref{eq:pressure_est_00}) holds true replacing $p_h-\pi_h$ by $p_h-p$.
\end{remark}

\section{A fully discrete method}\label{se:fully}

We now analyze the discretization of~\eqref{eq:galp} by the implicit Euler method. For this purpose, we consider  a partition~$0=t_0<t_1<\ldots< t_N=T$ of the interval~$[0,T]$, and for each time level we look for approximations~$\bU_h^n\approx \bu(t_n)$ in~$V_h^{\rm div}$ and~$p_h^n\approx p(t_n)$ in~$Q_h$,
satisfying
\begin{equation}
\label{eq:galp_fd_weak}
\begin{split}
(D_t\bU^n_h,\bv_h)+\nu (\nabla \bU_h^n,\nabla\bv_h)&+b(\bU_h,\bU_h^n,\bv_h)-(p_h^n,\nabla\cdot\bv_h^n)\\
&{}+(\nabla\cdot\bu_h^n,q_h)+\mu (\nabla\cdot \bU_h^n,\nabla\cdot\bv_h)=
({\boldsymbol f}(t_n),\bv_h),
\end{split}
\end{equation}
for all~$(\bv_h,q_h)\in (V_h\times Q_h)$,
for $n=1,\ldots,N$,
where $\bU_h^0\in V_h$ is given, and
$$
D_t\bU^n_h=\frac{\bU_h^n-\bU_h^{n-1}}{t_n-t_{n-1}}.
$$
In what follows we will take~$\bU_h^0=\bs_h(0)$ and consider for simplicity constant step sizes,
that is
$$
t_{n}-t_{n-1}=\Delta t,\quad n=1,\ldots,N.
$$
The changes for variable step sizes as well as for other consistent initial approximations are straightforward.
Also, other time integrators can be considered and the analysis can be carried out arguing essentially as in the next lines.

The existence of the approximation can be proved with the help of Brouwer's fixed point theorem as in \cite{Temam_NS}.
The approximations
$\bU_h^n$ satisfy
\begin{equation}
\label{eq:galp_fd}
D_t\bU^n_h+\nu A_h \bU_h^n+B_h(\bU_h,\bU_h^n)+\mu C_h \bU_h^n=
\Pi_h^{\rm div}{\boldsymbol f}(t_n),
\end{equation}
where we keep the notation of previous sections.

To obtain error bounds, we will use the following discrete Gronwall lemma that can be found in \cite{heyran4}.
\begin{lema}\label{le:discr_Gron}
Let $k$, $B$, and $a_n,b_n,c_n,\gamma_n$ be nonnegative numbers such that
$$
a_n+k\sum_{j=0}^n b_n\le k\sum_{j=0}^n\gamma_n a_n+k\sum_{j=0}^nc_n+B,\quad n\ge 1.
$$
Suppose that $k\gamma_n<1$, for all $n$, and set $\sigma_n=(1-k\gamma_n)^{-1}$. Then, the following bound holds
$$
a_n+k\sum_{j=0}^n b_n\le \exp\left(k\sum_{j=0}^n\sigma_j \gamma_j\right)\left(k\sum_{j=0}^n c_n +B\right),\quad n\ge 1.
$$
\end{lema}

\begin{lema}
\label{le:fd_stab} Fix $\gamma\in (0,1)$, and let $(\bw_h^n)_{n=0}^\infty$, $(\bt_{h,0}^n)_{n=1}^\infty$
and~$(\bt_{h,1}^n)_{n=1}^\infty$ be series in~$V_h^{\rm div}$ and let
$(t_{h,2}^n)_{n=1}^\infty$ a series in~$L^2(\Omega)$ satisfying
\begin{equation}
\label{eq:fd_aux1}
D_t\bw_h^n + \nu A_h \bw_h^n + B_h(\bw_h^n,\bw_h^n)+\mu C_h\bw_h^n =\Pi_h^{\rm div}\bff(t_n) + \nu A_h \bt_{h,0}^n + \bt_{h,1}^n + D_h t_{h,2}^n.
\end{equation}
Assume $\Delta t g_{h}^n < \gamma$ where
\begin{equation}
\label{eq:el_g_n}
g_{h}^n = 1 + 2\bigl\| \nabla \bw_h^n\bigr\|_\infty + \frac{\bigl\| \bw_h^n\bigr\|_\infty^2}{2\mu},\quad n=1,2,\ldots\,.
\end{equation}
There exists a positive constant~$C$ depending only on~$\gamma$ such that the following bound holds for the differences~$\be_h^n = \bw_h^n - \bU_h^n$:
\begin{eqnarray*}
\lefteqn{
\|\be_h^n\|_0^2+\sum_{j=1}^n\|\be_h^j-\be_h^{j-1}\|_0^2+
\Delta t\sum_{j=1}^n\left(\nu\|\nabla \be_h^j\|_0^2
+\mu \|\nabla \cdot \be_h^j\|_0^2\right)}  \\
 &\le& C\exp\left(\Delta t\sum_{j=1}^{n-1} g_j\right)\left(
\|\be_h^0\|_0^2+ \Delta t \sum_{j=1}^n \left(\nu\|\nabla \bt_{h,0}^j\|_0^2 +\|\bt_{h,1}^j\|_0^2
+\frac{1}{\mu} \| t_{h,2}^j\|_0^2 \right)
\right).\nonumber
\end{eqnarray*}
\end{lema}

\begin{proof}
A direct calculation shows that
$$
(\be_h^n-\be_n^{n-1},\be_h^n)=\frac{1}{2}\|\be_h^n\|_0^2-\frac{1}{2}\|\be_h^{n-1}\|_0^2+\frac{1}{2}\|\be_h^n-\be_h^{n-1}\|_0^2.
$$
so that
subtracting~(\ref{eq:galp_fd}) from~(\ref{eq:fd_aux1}), taking the inner
product in~$L^2(\Omega)$ with~$2\be_h^n$, adding
$0=2\mu(\nabla\cdot \bw_h^n,\nabla\cdot \be_h^n)-2\mu(\nabla\cdot(\bw_h^n-\bu(t_n)),\nabla\cdot \be_h^n)$ and  after some
rearrangements we have
\begin{eqnarray}
\label{eq:eu_aux_0}
\lefteqn{\frac{1}{\Delta t}\left(\|\be_h^n\|_0^2-\|\be_h^{n-1}\|_0^2+\|\be_h^n-\be_h^{n-1}\|_0^2\right)+\nu\|\nabla \be_h^n\|_0^2
+\frac{3}{2}\mu\|\nabla \cdot \be_h^n\|_0^2}
\nonumber\\
&\le&\nu\|\nabla\bt_{h,0}^n\|_0^2
 -2b(\be_h^n,\bw_h^n,\be_h^n) +\|\bt_{h,1}^n\|_0^2
 +\frac{2}{\mu}\|t_{h,2}^n\|_0^2
+\| \be_h^n\|_0^2 ,
\end{eqnarray}
where the product $(\be_h^n,B_h(\bw_h^n,\bw_h^n) - B_h(\bU_h^n,\bU_h^n))=
b(\bw_h^n,\bw_h^n,\be_h^n)-b(\bU_h^n,\bU_h^n,\be_h^n)$ has been treated as in~(\ref{eq:un_nl1}).
Arguing as in~(\ref{eq:un_nl2}) we may write
$$
b(\be_h^n,\bw_h^n,\be_h^n)
\le
\left(\|\nabla \bw_h^n\|_\infty +\frac{\|\bw_h^n\|_\infty^2}{4\mu}\right)\|\be_h^n\|_0^2+ \frac{\mu}{4}\|\nabla \cdot \be_h^n\|_0^2.
$$
Thus, multiplying by~$\Delta t$ in~(\ref{eq:eu_aux_0}) it follows that
\begin{equation}
\label{eq:eu_aux_1}
\|\be_h^n\|_0^2-\|\be_h^{n-1}\|_0^2+\|\be_h^n-\be_h^{n-1}\|_0^2+\nu\Delta t\|\nabla \be_h^n\|_0^2
+\mu\Delta t\|\nabla \cdot \be_h^n\|_0^2 \le  c_n+\Delta t g_h^n\|\be_h^n\|_0^2,
\end{equation}
where
\begin{equation*}
c_n=
\Delta t\left(\nu\|\nabla\bt_{h,0}^n\|_0^2
 +\|\bt_{h,1}^n\|_0^2
+\frac{2}{\mu}\|t_{h,2}^n\|_0^2
\right).
\end{equation*}
Adding the expression in~(\ref{eq:eu_aux_1}) to those corresponding to $n-1$, $n-2$, etc, down to~1, we have
\begin{eqnarray*}
\lefteqn{\hspace*{-7em}
\left(1-\Delta tg_h^n\right)\|\be_h^n\|_0^2+\sum_{j=1}^n\|\be_h^j-\be_h^{j-1}\|_0^2+
\Delta t \sum_{j=1}^n\left(\nu\|\nabla \be_h^j\|_0^2
+\mu\|\nabla \cdot \be_h^j\|_0^2\right)}
\nonumber
\\
&\le & \|\be_h^0\|_0^2+ \sum_{j=1}^n c_j+
 \sum_{j=1}^{n-1} \Delta t g_h^j\|\be_h^j\|_0^2.
\end{eqnarray*}
Since we are assuming that $\Delta tg_h^n \le \gamma$, we have $(1-\Delta tg_h^n )> 1-\gamma>0$ and the proof is finished by applying Lemma~\ref{le:discr_Gron}.
\end{proof}

\subsection{Error analysis in the regular case}
We apply Lemma~\ref{le:fd_stab} with $\bw_h^n=\bs_h(t_n)$, so
that
$$
\bt_{h,0}^n={\bf 0},\quad \bt_{h,1}^n=\Pi_h^{\rm div}({\boldsymbol\tau}_{1,1}^n+{\boldsymbol\tau}_{1,2}^n+
{\boldsymbol\tau}_2(t_n)),\quad t_{h,2}^n=\tau_3(t_n)+\tau_4(t_n),
$$
where ${\boldsymbol\tau}_2$, $\tau_3$ and~$\tau_4$ are those defined in~(\ref{eq:tau3}), and
\begin{align}
\label{eq:tau_11}
{\boldsymbol\tau}_{1,1}^n&=
\frac{(\bs_h(t_n)-\bu(t_h))-(\bs_h(t_{n-1})-\bu(t_{n-1}))}
{\Delta t}=\frac{1}{\Delta t}\!\int_{t_{n-1}}^{t_n}\partial_t(\bs_h(t)-\bu(t))\,dt,
\hphantom{\quad}\\
{\boldsymbol\tau}_{1,2}^n&=
\frac{\bu(t_n)-\bu(t_{n-1})}{\Delta t} -\partial_t\bu(t_n)= -
\frac{1}{\Delta t}\!\int_{t_{n-1}}^{t_n}(s-t_{n-1}) \partial_{tt}\bu(s)\,ds .
\label{tau_12}
\end{align}
We notice that in view of~(\ref{lasinfinito}) we have that
\begin{equation}
\label{eq:fd_L}
g_h^j
\le \hat L= 1 +C\max_{0\le t\le T}\left(2\left\|\nabla\bu(t)\right\|_\infty + \frac{\left\|\bu(t)\right\|_2^2}{2\mu}\right),
\quad j=1,\ldots, N,
\end{equation}
so that
\begin{equation}
\label{eq:fd_cota_exp1}
\exp\left(\Delta t\sum_{j=1}^{n-1}g_j\right) \le \exp\left(\hat Lt_n\right).
\end{equation}
We also have that~${\boldsymbol\tau}_2(t_n)$, $\tau_3(t_n)$, and~$\tau_4(t_n) $
have already been bounded in~(\ref{eq:cota_tau3_tau4}) and~(\ref{eq:cota_tau1_tau2}). Furthermore,
applying~Cauchy-Schwarz inequality in (\ref{eq:tau_11}) and the bound~(\ref{eq:cotanewpro})
with $j=k-1$ applied to~$\partial_t(\bs_h(t)-\bu(t))$ we have
$$
\| {\boldsymbol\tau}_{1,1}^n\|_0^2 \le C\frac{h^{2k}}{\Delta t}
\int_{t_{n-1}}^{t_n}\|\partial_t \bu\|_{k}^2\,dt
$$
and applying the Cauchy-Schwarz inequality in (\ref{tau_12})
$$
\| {\boldsymbol\tau}_{1,2}^n\|_0^2 \le C\Delta t
\int_{t_{n-1}}^{t_n}\|\partial_{tt} \bu\|_0^2\,dt.
$$
Thus, we have the following result.

\begin{Theorem}\label{th:fd_velocity_smooth} For $T>0$ let us assume for the solution $(\bu,p)$ of \eqref{eq:NSweak} that
$$\bu\in L^\infty(0,T,H^{k+1}(\Omega))\cap L^\infty(0,T,W^{1,\infty}(\Omega)),$$
$\bu(0)\in H^{\max\{2,k\}}(\Omega)$, $\partial_t\bu\in L^2(0,T,H^{k}(\Omega))$,  $\partial_{tt}\bu\in L^2(0,T,L^2(\Omega))$
and $p\in L^2(0,T,H^k(\Omega)/{\mathbb R})$ with $k\ge 1$.
Then, there exist positive constants $C_1$ and $C_2$ depending on
$$
\left\|\bu(0)\right\|_{k}^2+\int_0^T \| \partial_t\bu(t)\|_k^2 \,dt
+\max_{0\le t\le T} \left(\frac{\left\|p(t)\right\|_{H^k/{\mathbb R}}^2}{\mu}
+\left(\mu+\|\bu(t)\|_2^2\right)\| \bu(t)\|_{k+1}^2\right)
$$
and
$$
\int_0^T \|\partial_{tt}\bu(t)\|_0^2\,dt,
$$
respectively,
but none of them depending  on inverse powers of $\nu$, such that
the following bound holds for $\be_h=\bU_h^n-{\bs}_h(t_n)$
and~$1\le n\le N$
\begin{equation}\label{eq:fd_error_bound_velo}
\|\be_h^n\|_0^2+
\Delta t\sum_{j=1}^n\left(\nu\|\nabla \be_h^n\|_0^2
+\mu\|\nabla \cdot \be_h^n\|_0^2\right)
\le \exp\left(\hat Lt_n\right) \left(C_1h^{2k} + C_2(\Delta t)^2\right),
\end{equation}
where $\hat L$ is defined in~(\ref{eq:fd_L}).
\end{Theorem}

For the pressure, we can obtain error bounds by repeating the analysis
in~Section~\ref{se:pressure} with~$\partial_t\be_h$ replaced by
$D_t\be_h^n$ and the truncation error~${\boldsymbol\tau}_1$ by
${\boldsymbol\tau}_{1,1}^n + {\boldsymbol\tau}_{1,2}^n$. We observe, however, that instead of~(\ref{eq:auxB-12}) we now have
\begin{eqnarray*}
\lefteqn{
\|B(\bU_h^n,\bU_h^n)-B({\bs}_h(t_n),{\bs}_h(t_n))\|_{-1}}\\
& \le &
C\left\|\bu(t_n)\right\|_2
 \left(\|\be_h^n\|_0+\|\nabla \cdot \be_h^n\|_0\right) +
C\left( h^{-d/2}\left\|\be_h^n\right\|_0+h^{-1/2}\|\nabla \cdot \be_h^n\|_0
\right)\|\be_h^n\|_0.
\end{eqnarray*}
Now, the errors $\left\|\be_h^n\right\|_0$ and~$\|\nabla \cdot \be_h^n\|_0$,
as shown in~(\ref{eq:fd_error_bound_velo}) are bounded in terms of
powers of~$h$ and~$\Delta t$.  Thus, we have
the following result.
\begin{Theorem}\label{th:fd_pressure_smooth} In the conditions
of~Theorem~\ref{th:fd_velocity_smooth}, there exists a positive constant~$C$ such that the following bound holds
\begin{equation}
\beta_0^2\Delta t\sum_{j=1}^N \|p_h^n -\pi_h(t_n)\|_0^2 \le C\left(
1+\frac{(\Delta t)^2}{h^d}\right)\left(h^{2k}+(\Delta t)^2\right).
\label{eq:fd_cota_pres_ult}
\end{equation}
where in the case~$d=3$ the bound is valid for $k\ge 2$.
\end{Theorem}

\begin{remark}\label{re:dt/h} Observe that in~Theorem~\ref{th:fd_pressure_smooth} the presence of negative powers of~$h$ in~the error bound~\eqref{eq:fd_cota_pres_ult} does not
affect the convergence rate in the pressure whenever $\Delta t \le C h^{d/2}$ for any positive constant $C$. This condition will be automatically satisfied if we try to balance spatial and temporal discretization errors, since in that case we  would have to take $\Delta t \approx h^k$.

%

\end{remark}


\subsection{Error analysis without compatibility conditions}
We now assume that (\ref{eq:velo_h2_ass}) holds and that $M_{3}<+\infty$ and
$$
K_{3,0}=\int_{0}^T t\|\partial_{tt}\bu(t)\|_{0}^2\,dt <+\infty.
$$
In the case $k=1$ and~$d=2$ we will also assume $K_3<+\infty$. The cases
$k=1$ and $k\ge 2$ will be analyzed separately.

For $k=1$, and taking into account that~$\left\|\bu(t)\right\|_\infty\le \|\bu(t)\|_2$ and $\left\|\nabla\bu(t)\right\|_\infty
\le C\left\|\bu(t)\right\|_3 \le CM_3t^{-1/2}$, the analysis of the previous section is still valid with the following two changes. First we must
replace~(\ref{eq:fd_L}) and~(\ref{eq:fd_cota_exp1}) by
$g_h^j \le \hat L_{1} +2M_3t_j^{-1/2}$, for $j=1,\ldots,N$ where
\begin{equation}
\label{eq:fd_L2}
\hat L_{1} =1+\frac{M_2^2}{2\mu},
\end{equation}
so that
\begin{equation}
\label{eq:fd_cota_exp2}
\exp\left(\Delta t\sum_{j=1}^{n-1}g_h^j \right)\le \exp\left(\hat L_1t_n + 2M_3\sqrt{t_n}\right).
\end{equation}
The second and more relevant change is the estimation of~${\boldsymbol\tau}_{1,2}$, which now is
\begin{align*}
\| {\boldsymbol\tau}_{1,2}^n\|_0^2 &\le \frac{1}{\Delta t ^2}
\left(\int_{t_{n-1}}^{t_n} (t-t_n)\,dt\right)\left(\int_{t_{n-1}}^{t_n} (t-t_n)
\|\partial_{tt}\bu(t)\|_0^2\,dt\right)\\
&\le \frac{1}{2}\int_{t_{n-1}}^{t_n} t
\|\partial_{tt}\bu(t)\|_0^2\,dt.
\end{align*}
Thus, we have the following result.

\begin{Theorem}\label{th:fd_non_k1}
Fix $T>0$ and assume that the solution $(\bu,p)$ of (\ref{NS}) satisfies (\ref{eq:velo_h2_ass}), and that $M_3$, $K_3$ and~$K_{3,0}$ are finite. Assume linear finite element approximations in the velocity are used. Then,
\begin{itemize}
\item[i)] There exists a positive constant~$C_1$ depending on~$M_2$ and~$K_3$ but not on inverse powers of~$\nu$, such that
the following
bound holds for the error $\be_h^n=\bs_h(t_n)-\bU_h^n$, $1\le n\le N$:
\begin{eqnarray*}\lefteqn{\hspace*{-3em}
\|\be_h^n\|_0^2+\sum_{j=1}^n\|\be_h^j-\be_h^{j-1}\|_0^2+
\Delta t\sum_{j=1}^n\left(\nu\|\nabla \be_h^n\|_0^2
+\mu\|\nabla \cdot \be_h^n\|_0^2\right)}
\nonumber\\
&\le& \exp\left(\hat L_1t_n + 2M_3\sqrt{t_n}\right) \left(C_1h^{2} + K_{3,0}\Delta t\right),
\end{eqnarray*}
where $\hat L_1$ is defined in~(\ref{eq:fd_L2}).
\item[ii)] In the case $d=2$ then there exists a positive constant~$C_2$ depending on~$M_2$,
$M_3$, $\hat L_1$ and~$K_3$ but not on inverse powers of~$\nu$, such that
the following
bound holds:
$$
\beta_0^2\Delta t\sum_{j=1}^N \|p_h^n -\pi_h(t_n)\|_0^2 \le C_2
\left(1+\frac{\Delta t}{h}\right)\left(h^{2}+\Delta t\right).
$$
\end{itemize}
\end{Theorem}
\begin{remark} Let us observe that contrary to Theorem~\ref{th:fd_velocity_smooth} we have found a limitation in the rate of convergence of order $\mathcal O((\Delta t)^{1/2})$ in the temporal error. To our knowledge this paper is the first one in which error bounds with constants independent on $\nu$ are obtained for the fully discrete case without assuming nonlocal compatibility conditions for the Navier-Stokes equations. At present it is an open problem to find out if this~$\mathcal O((\Delta t)^{1/2})$ limitation could be avoided using a different technique of analysis.
\end{remark}

For $k\ge 2$, we will apply Lemma~\ref{le:fd_stab} to the differences
$$
\be_h^n=\bu_h^n - \bU_h^n,
$$
where
$$
\bu_h^n = \Pi^{\rm div}_h \bu(t_n) \quad n=0,1,\ldots,N.
$$
By projecting~(\ref{NS}) onto~$V_h^{\rm div}$ and adding and subtracting terms, it is easy to check that the projections
$\bu_h^n$ satisfy for $n=1,\ldots, N$,
$$
D_t\bu_h^n+\nu A_h \bu_h^n+B_h(\bu_h^n,\bu_h^n)+\mu C_h \bu_h^n=
\nu A_h\btt_{h,0}^n -\btt_{1,h}^n-D_h \tilde t_2^n,
$$
where, $ \btt_{h,0}^n=\bu_h^n-{\bs}_h(t_n)$, $\btt_{1,h}^n=\Pi_h^{\rm div}{\boldsymbol\tau}_{1,2}^n+\tilde{\boldsymbol\tau}_2^n$, and $\tilde t_2^n=\tau_3(t_n)+\tilde\tau_4^n$,
${\boldsymbol\tau}_{1,2}^n$ and $\tau_3$ being those
defined in~(\ref{tau_12}) and~(\ref{eq:tau3}),
and where
$$
\tilde{\boldsymbol\tau}_2=B_h(\bu_h^n,\bu_h^n)-B_h(\bu(t_n),\bu(t_n)),\quad
\tilde\tau_4^n=\mu\nabla \cdot(\bu_h^n-\bu(t_n)).
$$

We will need estimates in~$L^\infty(\Omega)$ of~$\bu_h^n$, which
are given by the following result.

\begin{lema}\label{le:est_inf} There is a constant~$C_D>0$ such that the following bounds hold:
\begin{align*}
\bigl\|\Pi_h^{\rm div}\bu\bigr\|_\infty &\le C_D\left\|\bu\right\|_2,\quad
\bigl\|\nabla\Pi_h^{\rm div}\bu\bigr\|_\infty
\le C_D\left\|\bu\right\|_3.
\end{align*}
\end{lema}
\begin{proof} We prove the second inequality for $d=3$, the case $d=2$ and the first inequality being proved
similarly. By adding $\pm I_h\bu$, using~(\ref{inv}) and~(\ref{cota_inter})
we have
$$
\bigl\|\nabla\Pi_h^{\rm div}\bu\bigr\|_\infty \le
C\left(h^{-(2+d)/2}\bigl\| (\Pi_h^{\rm div} -I_h)
\bu\bigr\|_0 + \left\|\nabla \bu\right\|_\infty\right).
$$
For the first term, by writing $(\Pi_h^{\rm div} -I_h)=
(\Pi_h^{\rm div} -I)+(I -I_h)$ and using~(\ref{eq:cota_pih}) with $j=2,3$
and~(\ref{cota_inter}) with $m=0$ and~$n=2,3$ we have
$$
h^{-(2+d)/2}\bigl\| (\Pi_h^{\rm div} -I_h)
\bu(t)\bigr\|_0 \le C(\left\|\bu\right\|_2\left\|\bu\right\|_3)^{1/2}
\le C\left\|\bu\right\|_3.
$$
Then, the result follows by estimating $\left\|\nabla \bu(t)\right\|_\infty$
with Sobolev's inequality~(\ref{sob1}).
\end{proof}

Thus, as with the case~$\bw_h^n=\bs_h(t_n)$, we also have that for~$\bw_h^n=\bu_h^n$ the value~$g_h^n$ defined
in~(\ref{eq:el_g_n}) satisfies the bound~(\ref{eq:fd_cota_exp2}). We also
observe that~$\tilde{\boldsymbol\tau}_{2}$ and~$\tilde\tau_4$ can be
estimated similarly to~(\ref{eq:auxtau2}) so that we can write
$$
\|\tilde{\boldsymbol\tau}_2^n\|_0\le C\left\|\bu(t_n)\right\|_2
\|\bu(t_n)-\bu_h^n\|_{1}\le C \left\|\bu\right\|_2\left\|\bu\right\|_{3}h^2
\le CM_2\frac{M_3}{t_n^{1/2}}h^2,
$$
and, similarly to~(\ref{eq:cota_tau3_tau4}),
$$
\|\tau_3^n+\tilde \tau_4^n\|_0^2 \le C(\|p(t_n)\|_{2}^2+
\mu^2\|\bu(t_n)\|_{3}^2)h^{4}\le C(1+\mu^2)\frac{M_3^2}{t_n}h^{4},
$$
and, applying~(\ref{eq:cota_pih}) and (\ref{eq:cotanewpro})
$$
\|\nabla\btt_{h,0}^n\|_0^2 \le C \|\bu(t_n)\|_3^2h^4 \le C\frac{M_3^2}{t_n}h^4.
$$
Then, noticing that
$$
\Delta t\sum_{j=1}^{n} \frac{1}{t_j}=1+\sum_{j=2}^{n-1}\frac{1}{j}
\le 1+\log\left(\frac{t_{n+1}}{\Delta t}\right).
$$
we have
\begin{eqnarray*}
\lefteqn{
\Delta t \sum_{j=1}^n \left(\nu\|\nabla \bt_{h,0}^j\|_0^2 +\|\bt_{h,1}^j\|_0^2
+\frac{1}{\mu} \| t_{h,2}^j\|_0^2\right)}\\
&\le&
C\left(\log\left(\frac{t_{n+1}}{\Delta t}\right)M_3^2
\left(\nu+M_2^2+ \frac{1}{\mu}+
\mu\right)h^4 +K_{3,0}\Delta t\right).
\end{eqnarray*}
We conclude with the following result.
\begin{Theorem}\label{th:fd_non_k2}
For $T>0$, assuming the solution $(\bu,p)$ of (\ref{NS}) satisfies (\ref{eq:velo_h2_ass}), and that $M_3$, $K_3$ and~$K_{3,0}$ are finite.
Assume that piecewise approximations in the velocity of degree $k\ge 2$ are used. Then,
the following
bound holds for the error $\be_h^n=\Pi_h^{\rm div}\bu(t_n)-\bU_h^n$, $1\le n\le N$
\begin{eqnarray}
\lefteqn{
\|\be_h^n\|_0^2+\sum_{j=1}^n\|\be_h^j- \be_h^{j-1}\|_0^2+
\Delta t\sum_{j=1}^n\left(\nu\|\nabla \be_h^n\|_0^2
+\mu\|\nabla \cdot \be_h^n\|_0^2\right)}
\nonumber\\
&\le& \exp\left(\hat L_1t_n + 2M_3\sqrt{t_n}\right) \left(C_3\log\left(\frac{t_{n+1}}{\Delta t}\right)h^{4} + K_{3,0}\Delta t\right),
\label{eq:fd_ebound_velo_k2}
\end{eqnarray}
where $\hat L_1$ is defined in~(\ref{eq:fd_L2}), and $C_3$ depends on~$M_2$ and~$M_3$ but not on inverse powers of~$\nu$.
\end{Theorem}

For the pressure, we take inner product of the first equation in~(\ref{NS})
with~$\bv_h\in V_h$, subtract~(\ref{eq:galp_fd_weak}), add~$\pm (\pi_h(t_n),\nabla\cdot\bv_h)$ and use the
inf-sup condition, to obtain (after some rearrangements)
\begin{eqnarray}\label{eq:fd_press_aux}
\beta_0\|p_h^n-\pi_h(t_n)\|_0&\le& \nu\|\nabla \be_h^n\|_0+\|B(\bu_h^n,\bu_h^n)-B({\bU}_h^n,{\bU}_h^n)\|_{-1}+ \|D_t\be_h^n\|_{-1}\nonumber\\
&&+\mu\|\nabla\cdot \be_h^n\|_0 +\|\Pi_h^{\rm div}{\boldsymbol\tau}_{1,2}^n\|_{-1}+\|(I-\Pi_h^{\rm div})\partial_t\bu(t_n)\|_{-1}\nonumber
\\
&&+\|\tilde{\boldsymbol\tau}_2\|_{-1}+\|\tau_3\|_0+\|\tilde\tau_4\|_0+\|l_h(t_n)\|_0,
\end{eqnarray}
where $\be_h^n=\bu_h^n-\bU_h^n$. As in~Section~\ref{se:pressure},
we estimate
$\|\Pi_h^{\rm div}{\boldsymbol\tau}_{1,2}^n\|_{-1}\le C
\|\Pi_h^{\rm div}{\boldsymbol\tau}_{1,2}^n\|_{0}$
and~$\|\tilde{\boldsymbol\tau}_2\|_{-1}\le
C\|\tilde{\boldsymbol\tau}_2\|_{0}$. Using~(\ref{eq:cotanewpropre}) with
$j=1$ we have~$\|l_h(t_n)\|_0\le \nu h^2\|\bu(t_n)\|_2$. Also, repeating
the arguments that lead from~(\ref{eq:auxB-1}) to~(\ref{eq:auxB-12}) with
$\bs_h$ replaced~by~$\Pi_h^{\rm div}\bu$, and in view of~Lemma~\ref{le:est_inf}, we have
\begin{eqnarray*}
\|B(\bu_h^n,\bu_h^n)-B(\bU_h^n,\bU_h^n)\|_{-1}
&\le &
C\left\|\bu(t_n)\right\|_2
 \left(\|\be_h^n\|_0+\|\nabla \cdot \be_h^n\|_0\right) \\
 &&+
C\left( h^{-d/2}\left\|\be_h^n\right\|_0+h^{-1/2}\|\nabla \cdot \be_h^n\|_0
\right)\|\be_h^n\|_0.\nonumber
\end{eqnarray*}

We now estimate~$\|(I-\Pi_h^{\rm div})\partial_t\bu(t_n)\|_{-1}$. For
$\boldsymbol\phi\in H^1_0(\Omega)^d$, we use the Leray decomposition
$\boldsymbol\phi=\Pi\boldsymbol\phi + \nabla q$, and recall that
$\|\Pi\boldsymbol\phi\|_1\le C\|\boldsymbol\phi\|_1$ and
$\|q\|_2\le C\|\boldsymbol\phi\|_1$. We notice that
\begin{align}
\left((I-\Pi_h^{\rm div})\partial_t\bu,\nabla q\right)&=
-\left(\nabla\cdot (I-\Pi_h^{\rm div})\partial_t\bu,q\right)=
\left(\nabla\cdot (I-\Pi_h^{\rm div})\partial_t\bu,q-\pi_hq\right)
\nonumber\\
&=\left((I-\Pi_h^{\rm div})\partial_t\bu,\nabla (q-\pi_h q)\right)
\le Ch^2\|\partial_t\bu\|_1\|\boldsymbol\phi\|_1,
\label{eq:fd_aux3}
\end{align}
where in the last inequality we have applied~(\ref{eq:cota_pih}) with $j=0$ and~(\ref{eq:pressurel2})
with $m=1$ and~$j=1$.
We also have,
$$
\left((I-\Pi_h^{\rm div})\partial_t\bu,\Pi\boldsymbol\phi\right)=
\left((I-\Pi_h^{\rm div})\partial_t\bu,(I-\Pi_h^{\rm div})\Pi\boldsymbol\phi\right),
$$
so that applying~(\ref{eq:cota_pih}) with $j=0$, and together with~(\ref{eq:fd_aux3}), it easily follows that
\begin{equation}
\label{eq:est_1_ut}
\|(I-\Pi_h^{\rm div})\partial_t\bu(t_n)\|_{-1}\le Ch^2\|\partial_t\bu(t_n)\|_1.
\end{equation}
Finally, arguing as~Section~\ref{se:pressure} the term~$\|D_t\be_h^n\|_{-1}$ can be bounded by the terms on the the right-hand side of~(\ref{eq:fd_press_aux}) except itself, so that we can conclude with
the following result.

\begin{Theorem}\label{th:fd_press_k2} In the conditions of~Theorem~\ref{th:fd_non_k2}, there
exists a constant~$C>0$ not depending on inverse powers of~$\nu$ such the following bound holds,
$$
\beta_0^2\Delta t\sum_{j=1}^N \|p_h^n -\pi_h(t_n)\|_0^2 \le Cr(t_n,h,\Delta t)\left(1+ \frac{r(t_n,h,\Delta t)}{h^d}\right),
$$
where $r(t_n,h,\Delta t)$ is the right-hand side of~(\ref{eq:fd_ebound_velo_k2}).
\end{Theorem}

\begin{remark} As in Remark~\ref{re:dt/h}, we observe that if
the two sources of error (temporal and spatial) are to be balanced
in~(\ref{eq:fd_ebound_velo_k2}) at the final time $t_N=T$, then $\Delta t\approx h^4\log(N)$. Thus, $h^{-d} \Delta t =\mathcal O(h^{4-d}\log(N))$, so that the
presence of negative powers of~$h$ in the error bound in
Theorem~\ref{th:fd_press_k2} does not alter the convergence rate, and
the error is~$\mathcal O(h^4\log(N)+\Delta t)$.
\end{remark}

\section{Numerical studies}\label{sec:numres}

In this section, numerical studies will be presented that support the analytical results with respect to the  order
of convergence and the independence of the errors of $\nu$. As usual for such purposes, an example with a known solution is considered.

Let $\Omega=(0,1)^2$ and $T=5$, then the Navier--Stokes equations \eqref{NS} were considered
with the prescribed solution
\begin{eqnarray*}
\bu &=& \cos(t)\begin{pmatrix} \sin(\pi x-0.7)\sin(\pi y + 0.2) \\
\cos(\pi x-0.7)\cos(\pi y + 0.2)
\end{pmatrix}, \\
p &=& \cos(t)(\sin(x)\cos(y)+(\cos(1)-1)\sin(1)).
\end{eqnarray*}
It is clear that examples constructed in this way satisfy the nonlocal compatibility condition.
The simulations were performed
for the $P_2/P_1$ pair of finite element spaces on a regular triangular grid
consisting on the coarsest level~$0$ of two mesh cells (diagonal from lower left to upper right). The number of degrees of freedom for velocity/pressure
on level~$3$ is $578/81$ and on level~$8$ it is $526338/66049$. As temporal discretization, the Crank--Nicolson scheme was used.
The grad-div stabilization parameter was chosen to be $\mu = 0.25$ in all simulations,
see \cite{los_cuatro_oseen} for a motivation of this specific choice.
In each discrete time, the fully nonlinear problem was solved. The simulations
were performed with the code MooNMD \cite{JM04}.

\begin{figure}[t!]
\centerline{\includegraphics[width=0.45\textwidth]{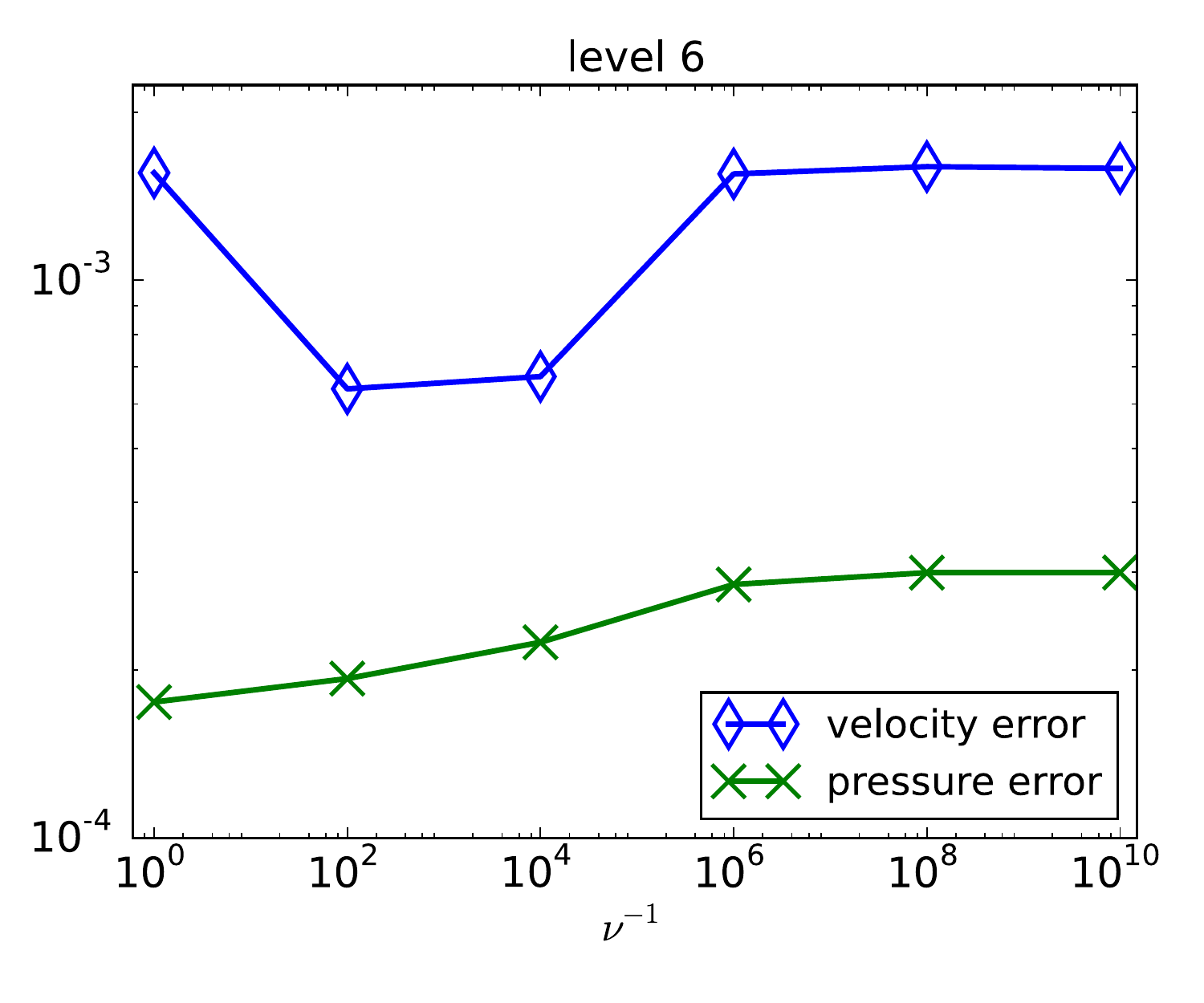}\hspace*{1em}
\includegraphics[width=0.45\textwidth]{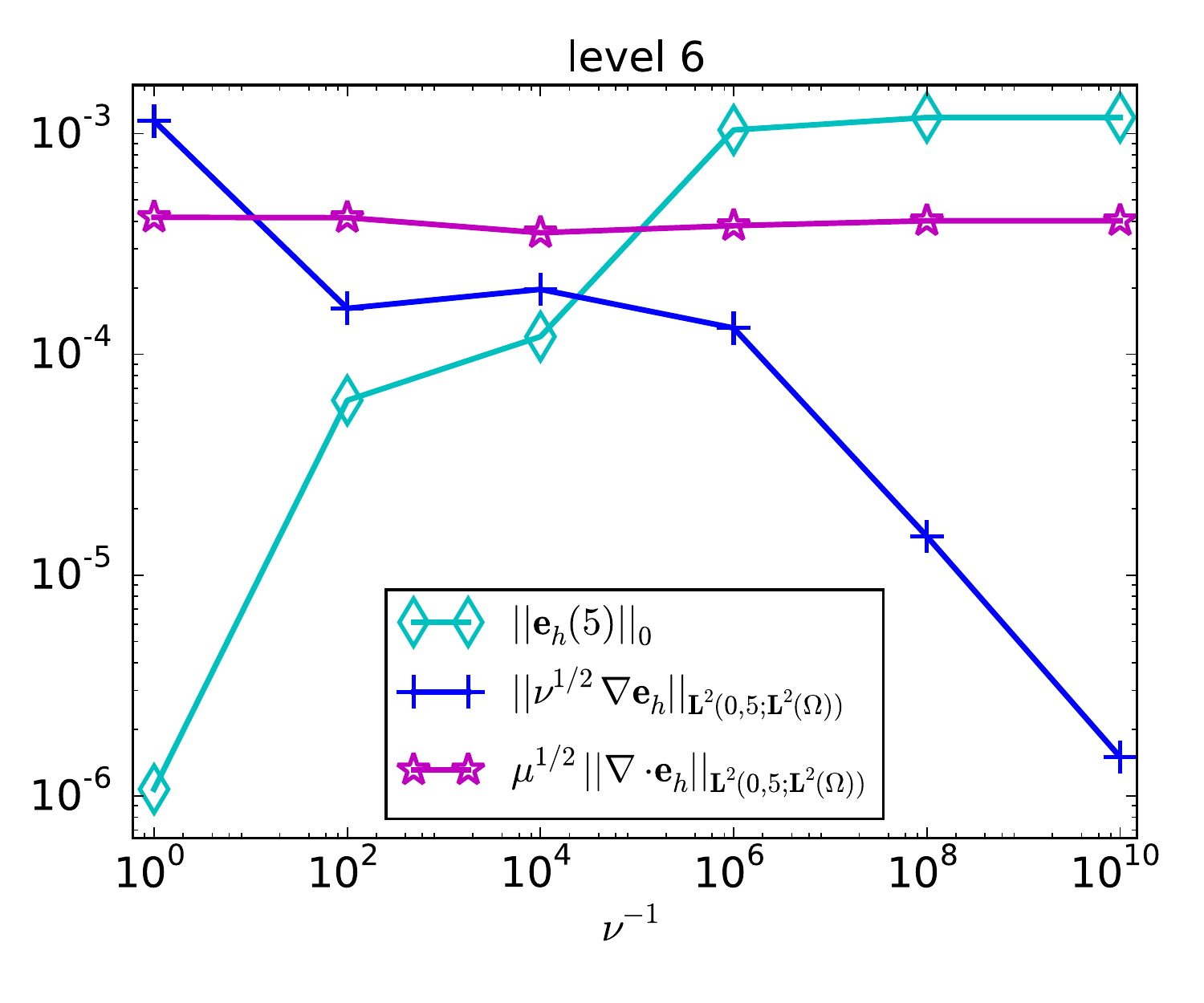}
}
\caption{Left: numerical results for the velocity error, left-hand side of estimate \eqref{eq:error_bound_velo},
and the pressure error, integral term on the left-hand side of  estimate \eqref{eq:pressure_est} with $p-p_h$:
different values of $\nu$ for a fixed spatial grid.  Right:
individual contributions of the left-hand side of \eqref{eq:error_bound_velo}.}
\label{fig:numres}
\end{figure}

Results of the numerical studies are presented in Figs.~\ref{fig:numres} and~\ref{fig:numres_1}. For the
simulations on level~$6$ with different values of $\nu$, Fig.~\ref{fig:numres}, the equidistant time step
$0.0625$ was used in the Crank--Nicolson scheme. It can be clearly seen that the
velocity and pressure errors, which were bounded in the analysis, are independent of
$\nu$. Considering the individual contributions of the velocity error, one can
observe that in particular the norm of the divergence is almost the same for all
values of $\nu$.

For the simulations with constant $\nu$ on a sequence of grids, the smaller time
steps $0.002$ and
$0.001$ were used. Because the curves for both time steps are almost on top of each other,
see Fig.~\ref{fig:numres_1},
it can be concluded that the temporal error is negligible. The pressure error
decreases somewhat faster than predicted by the theory with an order  of nearly $2.9$.
Also the velocity error decreases faster on coarse grids because on these grids
the contribution $\|\be_h(5)\|_0$ dominates which is reduced by a higher order than
two, compare the right picture
of Fig.~\ref{fig:numres_1}. But on finer grids, the predicted second order
convergence can be seen.

\begin{figure}[t!]
\centerline{
\includegraphics[width=0.45\textwidth]{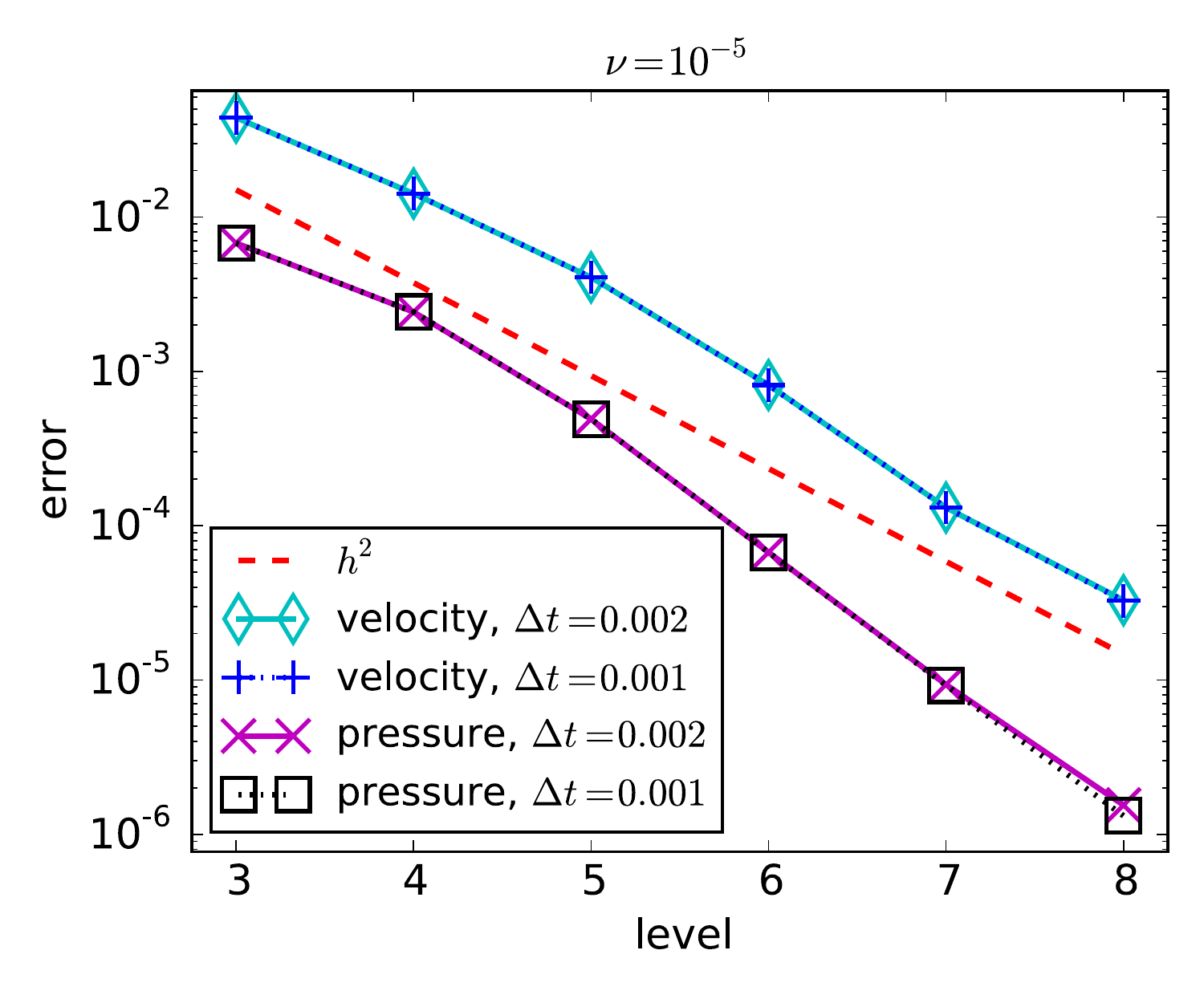}
\hspace*{1em}
\includegraphics[width=0.45\textwidth]{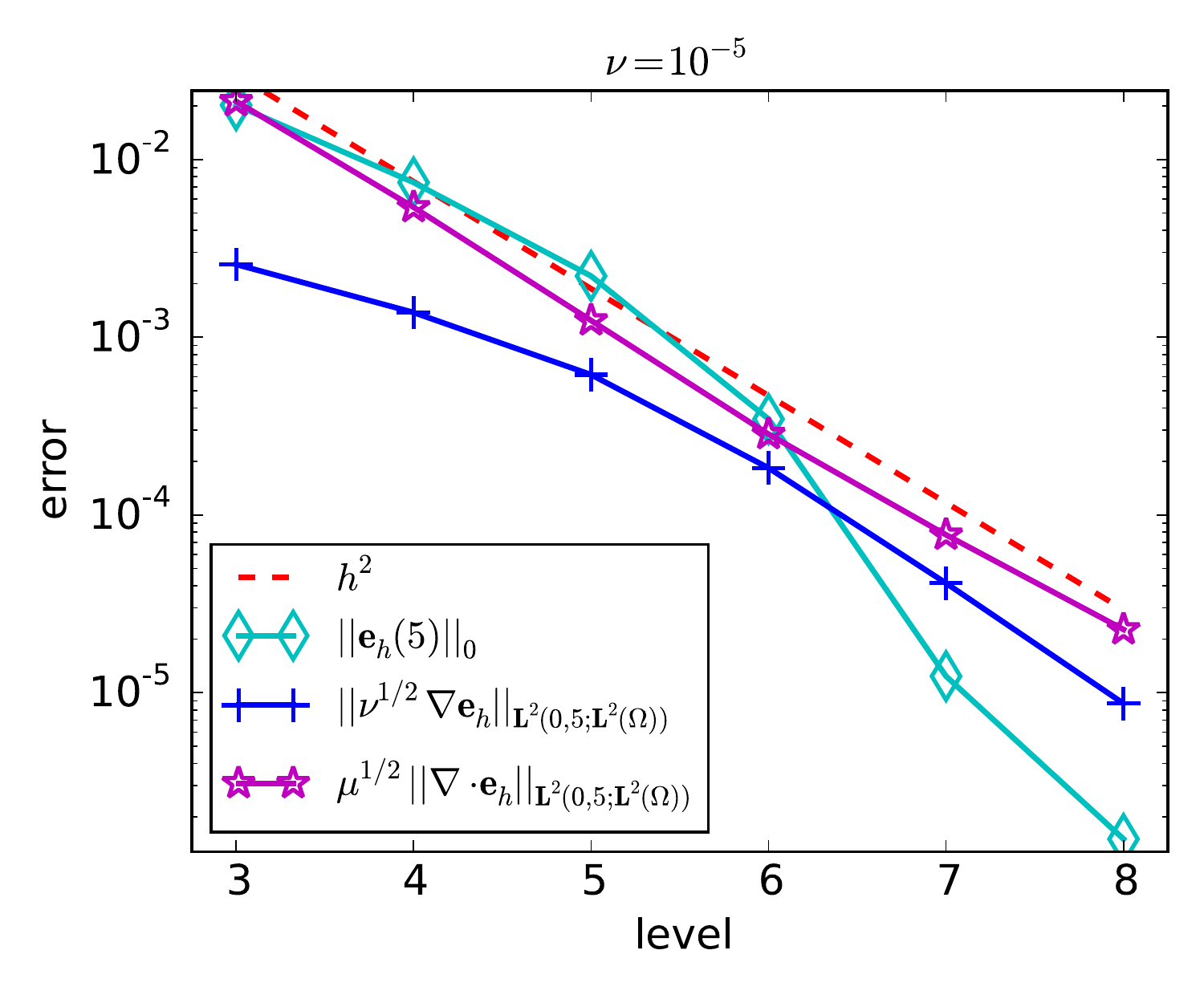}}
\caption{Numerical results for the velocity error, left-hand side of estimate \eqref{eq:error_bound_velo},
and the pressure error, integral term on the left-hand side of  estimate \eqref{eq:pressure_est} with $p-p_h$.
Left: different grid levels and different lengths of the time step for fixed $\nu$
(the lines for the time steps $0.001$ and $0.002$ are almost on top of each other).
Right: the individual terms of the left-hand side of \eqref{eq:error_bound_velo}.}
\label{fig:numres_1}
\end{figure}

\section{Summary}
Inf-sup stable finite element discretizations are considered to approximate the evolutionary
Navier--Stokes equations. The Galerkin finite element method is augmented with a grad-div stabilization term. It had been reported in the literature \cite{JK10,RL10} that
stable simulations were obtained
in the computation of turbulent flows using exclusively grad-div stabilization.
This observation is the motivation of the present paper. Error bounds for the Galerkin
plus grad-div stabilization method were derived, both for the continuous-in-time case
and a fully discrete scheme. The error constants do not depend on inverse
powers of $\nu$, although they depend on norms of the solution that are assumed to be bounded.
The paper extends a previous work by the same authors \cite{los_cuatro_oseen}, where
the evolutionary Oseen equations were considered. The analysis covers both the case in which
the solution is assumed to be smooth and the practically relevant situation
in which nonlocal compatibility conditions are not satisfied and, hence, the derivatives of the solution cannot be assumed to be bounded up to~$t=0$. To the best of our knowledge,
this paper  is the first one where this breakdown of regularity at $t=0$ has been taking into account
to analyze the effect of the grad-div stabilization. Related works like \cite{Lube_etalNS,Burman_Fer_numer_math,Burman_cmame2015} assume
that the solution satisfies nonlocal compatibility conditions.
The present paper also seems to be the first one where error bounds with constants independent of $\nu$
are obtained for a fully discrete method for the Navier--Stokes equations
without assuming nonlocal compatibility conditions.

\end{document}